\documentclass[a4paper,11pt]{article}
\usepackage{amsmath}
\usepackage{amsthm}
\usepackage{amssymb}
\usepackage{amscd}
\usepackage{graphicx}
\usepackage{epsfig}
\usepackage{color}

\usepackage{latexsym}
\usepackage{amsfonts}
\input xy
\xyoption{all} \tolerance=500

\newtheorem{thm}{Theorem}[section]
\newtheorem{prop}[thm]{Proposition}
\newtheorem{lem}[thm]{Lemma}
\newtheorem{cor}[thm]{Corollary}

\theoremstyle{definition}
\newtheorem{example}[thm]{Example}
\newtheorem{remark}[thm]{Remark}
\newtheorem{definition}[thm]{Definition}

\font\black=cmbx10 \font\sblack=cmbx7 \font\ssblack=cmbx5 \font\blackital=cmmib10  \skewchar\blackital='177
\font\sblackital=cmmib7 \skewchar\sblackital='177 \font\ssblackital=cmmib5 \skewchar\ssblackital='177
\font\sanss=cmss12 \font\ssanss=cmss8 scaled 900 \font\sssanss=cmss8 scaled 600 \font\blackboard=msbm10
\font\sblackboard=msbm7 \font\ssblackboard=msbm5 \font\caligr=eusm10 \font\scaligr=eusm7 \font\sscaligr=eusm5

\font\bsymb=cmsy10 scaled\magstep2
\def\all#1{\setbox0=\hbox{\lower1.5pt\hbox{\bsymb
       \char"38}}\setbox1=\hbox{$_{#1}$} \box0\lower2pt\box1\;}
\def\exi#1{\setbox0=\hbox{\lower1.5pt\hbox{\bsymb \char"39}}
       \setbox1=\hbox{$_{#1}$} \box0\lower2pt\box1\;}

\def\tx#1{{\fam0\relax#1}}

\newfam\bifam
\textfont\bifam=\blackital \scriptfont\bifam=\sblackital \scriptscriptfont\bifam=\ssblackital

\newfam\blfam
\textfont\blfam=\black \scriptfont\blfam=\sblack \scriptscriptfont\blfam=\ssblack

\newfam\bbfam
\textfont\bbfam=\blackboard \scriptfont\bbfam=\sblackboard \scriptscriptfont\bbfam=\ssblackboard

\newfam\ssfam
\textfont\ssfam=\sanss \scriptfont\ssfam=\ssanss \scriptscriptfont\ssfam=\sssanss
\def\sss#1{{\fam\ssfam\relax#1}}

\newfam\clfam
\textfont\clfam=\caligr \scriptfont\clfam=\scaligr \scriptscriptfont\clfam=\sscaligr

\def\pmb#1{\setbox0\hbox{${#1}$} \copy0 \kern-\wd0 \kern.2pt \box0}
\def\pmbb#1{\setbox0\hbox{${#1}$} \copy0 \kern-\wd0
      \kern.2pt \copy0 \kern-\wd0 \kern.2pt \box0}
\def\pmbbb#1{\setbox0\hbox{${#1}$} \copy0 \kern-\wd0
      \kern.2pt \copy0 \kern-\wd0 \kern.2pt
    \copy0 \kern-\wd0 \kern.2pt \box0}
\def\pmxb#1{\setbox0\hbox{${#1}$} \copy0 \kern-\wd0
      \kern.2pt \copy0 \kern-\wd0 \kern.2pt
      \copy0 \kern-\wd0 \kern.2pt \copy0 \kern-\wd0 \kern.2pt \box0}
\def\pmxbb#1{\setbox0\hbox{${#1}$} \copy0 \kern-\wd0 \kern.2pt
      \copy0 \kern-\wd0 \kern.2pt
      \copy0 \kern-\wd0 \kern.2pt \copy0 \kern-\wd0 \kern.2pt
      \copy0 \kern-\wd0 \kern.2pt \box0}

\mathchardef\za="710B  
\mathchardef\zb="710C  
\mathchardef\zg="710D  
\mathchardef\zd="710E  
\mathchardef\zve="710F 
\mathchardef\zz="7110  
\mathchardef\zh="7111  
\mathchardef\zvy="7112 
\mathchardef\zi="7113  
\mathchardef\zk="7114  
\mathchardef\zl="7115  
\mathchardef\zm="7116  
\mathchardef\zn="7117  
\mathchardef\zx="7118  
\mathchardef\zp="7119  
\mathchardef\zr="711A  
\mathchardef\zs="711B  
\mathchardef\zt="711C  
\mathchardef\zu="711D  
\mathchardef\zvf="711E 
\mathchardef\zq="711F  
\mathchardef\zc="7120  
\mathchardef\zw="7121  
\mathchardef\ze="7122  
\mathchardef\zy="7123  
\mathchardef\zf="7124  
\mathchardef\zvr="7125 
\mathchardef\zvs="7126 
\mathchardef\zf="7127  
\mathchardef\zG="7000  
\mathchardef\zD="7001  
\mathchardef\zY="7002  
\mathchardef\zL="7003  
\mathchardef\zX="7004  
\mathchardef\zP="7005  
\mathchardef\zS="7006  
\mathchardef\zU="7007  
\mathchardef\zF="7008  
\mathchardef\zW="700A  

\newcommand{\be}{\begin{equation}}
\newcommand{\ee}{\end{equation}}
\newcommand{\ra}{\rightarrow}

\newcommand{\bea}{\begin{eqnarray}}
\newcommand{\eea}{\end{eqnarray}}
\newcommand{\beas}{\begin{eqnarray*}}
\newcommand{\eeas}{\end{eqnarray*}}
\def\*{{\textstyle *}}
\newcommand{\R}{{\mathbb R}}

\newcommand{\C}{{\mathbb C}}

\newcommand{\we}{\wedge}
\newcommand{\nn}{\nonumber}
\newcommand{\ot}{\otimes}

\newcommand{\pa}{\partial}
\newcommand{\ti}{\times}
\newcommand{\Li}{{\cal L}}
\newcommand{\ka}{\mathbb{K}}

\newcommand{\ad}{{\rm ad}}

\newcommand{\X}{{\cal X}}

\newcommand{\Ll}{{\pounds}}

\def\ran{\rangle}

\def\Hom{\sss{Hom}}
\def\Der{\sss{Der}}
\def\End{\sss{End}}
\def\Lin{\sss{Lin}}

\def\op{\oplus}

\def\cA{{\cal A}}

\def\cD{{\cal D}}

\def\cE{{\cal E}}
\def\cL{{\cal L}}

\def\cT{{\cal T}}

\def\cX{{\cal X}}

\def\wh{\widehat}

\def\Sec{\sss{Sec}}

\def\cim{{C^\infty(M)}}

\def\la{\langle}
\def\ran{\rangle}

\def\sD{{\sss D}}

\def\sT{{\sss T}}

\def\xd{\tx{d}}

\def\rk{\operatorname{rk}}

\newdir{|>}{%
!/4.5pt/@{|}*:(1,-.2)@^{>}*:(1,+.2)@_{>}}

\def\bcE{{\bar\cE^0}}

\newcommand{\ope}[1]{\!\!\mathop{\rm ~#1}\nolimits}

\begin{document}
\title{{\bf The Supergeometry of Loday Algebroids}}
\date{}
\author{Janusz Grabowski\thanks{The research of J. Grabowski was
supported by the Polish Ministry of Science and Higher Education
under the grant N N201 416839.}, David Khudaverdyan, and Norbert
Poncin\thanks{The research of N. Poncin was supported by Grant
GeoAlgPhys 2011-2013 awarded by the University of Luxembourg.}}

\maketitle

\begin{abstract} A {new} concept of {\it Loday algebroid} (and its pure algebraic version -- {\it Loday pseudoalgebra}) is
proposed and discussed in comparison with other similar structures
present in the literature. The structure of a Loday pseudoalgebra
and its natural reduction to a Lie pseudoalgebra is studied.
Further, Loday algebroids are interpreted as homological vector
fields on a `supercommutative manifold' {associated with a shuffle product} and the corresponding
Cartan calculus is introduced. Several examples, including Courant
algebroids, Grassmann-Dorfman and twisted Courant-Dorfman
brackets, as well as algebroids induced by Nambu-Poisson
structures, are given.
\end{abstract}

\vspace{2mm} \noindent {\bf MSC 2000}: 53D17, 58A50, 17A32,
17B66\medskip

\noindent{\bf Keywords}: Algebroid, pseudoalgebra, Loday algebra,
Courant bracket, supercommutative manifold, homological vector
field, Cartan calculus

\section{Introduction}

The concept of {\it Dirac structure}, proposed by Dorfman \cite{Do} in the Hamiltonian framework of integrable
evolution equations and defined in \cite{Co} as an isotropic subbundle of the Whitney sum $\cT M=\sT
M\oplus_M\sT^\ast M$ of the tangent and the cotangent bundles and satisfying some additional conditions,
provides a geometric setting for Dirac's theory of constrained mechanical systems. To formulate the
integrability condition defining the Dirac structure, Courant \cite{Co} introduced a natural
skew-symmetric bracket operation on sections of \ $\cT M$. The Courant bracket does not satisfy the Leibniz
rule with respect to multiplication by functions nor the Jacobi identity. These defects disappear upon
restriction to a Dirac subbundle because of the isotropy condition. Particular cases of Dirac structures are
graphs of closed 2-forms and  Poisson bivector fields on the manifold $M$.

The nature of the Courant bracket itself remained unclear until several years later when it was observed by
Liu, Weinstein and Xu \cite{LWX} that $\cT M$ endowed with the Courant bracket plays the role of a `double'
object, in the sense of Drinfeld \cite{Dr}, for a pair of Lie algebroids (see \cite{Mac}) over $M$. Let us
recall that, in complete analogy with Drinfeld's Lie bialgebras, in the category of Lie algebroids there also
exist `bi-objects', Lie bialgebroids, introduced by Mackenzie and Xu \cite{MX} as linearizations of Poisson
groupoids. On the other hand, every Lie bialgebra has a double which is a Lie algebra. This is not so for
general Lie bialgebroids. Instead, Liu, Weinstein and Xu \cite{LWX} showed that the double of a Lie
bialgebroid is a more complicated structure they call a {\it Courant algebroid}, $\cT M$ with the Courant
bracket being a special case.

There is also another way of viewing Courant algebroids as a generalization of Lie algebroids. This requires a
change in the definition of the Courant bracket and considering an analog of the non-antisymmetric Dorfman
bracket \cite{Do}, so that the traditional Courant bracket becomes the skew-symmetrization of the new one
\cite{Roy}. This change replaces one of the defects with another one: a version of the Jacobi identity is
satisfied, while the bracket is no longer skew-symmetric. Such algebraic structures have been introduced by
Loday \cite{Lo} under the name {\it Leibniz algebras}, but they are nowadays also often called {\it Loday
algebras}. Loday algebras, like their skew-symmetric counterparts -- Lie algebras -- determine certain
cohomological complexes, defined on tensor algebras instead of Grassmann algebras. {Canonical examples of Loday algebras arise often as {\em derived brackets} introduced by Kosmann-Schwarzbach \cite{K-S,YKS}.}

Since Loday brackets, like the Courant-Dorfman bracket, appear naturally in Geometry and Physics in the form
of `algebroid brackets', i.e. brackets on sections of vector bundles, there were several attempts to formalize
the concept of {\it Loday} (or {\it Leibniz}) {\em algebroid} (see e.g. \cite{Ba,BV,G1,GM,ILMP,Ha,HM,KS,MM,SX,Wa}).
We prefer the terminology {\em Loday algebroid} to distinguish them from other {\em general algebroid}
brackets with both anchors (see \cite{GU}), called sometimes {\em Leibniz algebroids} or {\em Leibniz
brackets} and used recently in Physics, for instance, in the context of nonholonomic constraints
\cite{GG,GG1,GGU,GLMM,OPB}. Note also that a Loday algebroid is the horizontal categorification of a Loday algebra; vertical categorification would lead to Loday $n$-algebras, which are tightly related to truncated Loday infinity algebras, see \cite{AP10}, \cite{KMP11}.

The concepts of Loday algebroid we found in the literature do not seem to be exactly appropriate. The notion
in \cite{G1}, which assumes the existence of both anchor maps, is too strong and admits no real new examples,
except for Lie algebroids and bundles of Loday algebras. The concept introduced in \cite{SX} requires a
pseudo-Riemannian metric on the bundle, so it is too strong as well and does not reduce to a Loday algebra
when we consider a bundle over a single point, while the other concepts \cite{Ha,HM,ILMP,KS,MM,Wa}, assuming
only the existence of a left anchor, do not put any differentiability requirements for the first variable, so
that they are not geometric and too weak (see Example \ref{e1}). Only in \cite{Ba} one considers some Leibniz algebroids with local brackets.

The aim of this work is to propose a modified concept of Loday algebroid in terms of an operation on
sections of a vector bundle, as well as in terms of a homological vector field of a supercommutative manifold.
We put some minimal requirements that a proper concept of Loday algebroid should satisfy. Namely, the
definition of Loday algebroid, understood as a certain operation on sections of a vector bundle $E$,
\begin{itemize}
\item should reduce to the definition of Loday algebra in the case when $E$ is just a vector space;

\item should contain the Courant-Dorfman bracket as a particular example;

\item should be as close to the definition of Lie algebroid as possible.
\end{itemize}

We propose a definition satisfying all these requirements and
including all main known examples of Loday brackets with geometric
origins. Moreover, we can interpret our Loday algebroid structures
as homological vector fields on a supercommutative manifold; this
opens, like in the case of Lie algebroids, new horizons for a
geometric understanding of these objects and of their possible
`higher generalizations' \cite{BP12}.
{This supercommutative manifold is associated with a superalgebra of differential operators{, whose multiplication is a supercommutative shuffle product}.

Note that we cannot work with the supermanifold {$\Pi E$} like in the case of a Lie algebroid on $E$, since the Loday coboundary operator rises the degree of {a differential operator,} even for Lie algebroid brackets. For instance, the Loday differential associated with the standard bracket of vector fields produces the Levi-Civita connection out of a Riemannian metric \cite{Ldd1}. However, the Levi-Civita connection $\nabla_XZ$ is no longer a tensor, as it is of the first-order with respect to $Z$. Therefore, instead of the Grassmann algebra $\Sec(\we E^*)$ of `differential forms', which are zero-degree skew-symmetric multidifferential operators on $E$, we are forced to consider, not just the tensor algebra of sections of $\oplus_{k=0}^\infty(E^*)^{\otimes k}$, but the algebra $\cD^\bullet(E)$ spanned by all multidifferential operators
$$D:\Sec(E)\times\cdots\times\Sec(E)\to C^\infty(M)\,.$$
However, to retain the supergeometric flavor, we can reduce ourselves to a smaller subspace $\sD^\bullet(E)$ of $\cD^\bullet(E)$, which is a subalgebra with respect to the canonical supercommutative shuffle product and is closed under the Loday coboundary operators associated with the Loday algebroids we introduce. This interesting observation deserves further investigations that we postpone to a next paper.

We should also make clear that, although the algebraic structures in question have their roots in Physics (see the papers on Geometric Mechanics mentioned above), we do not propose in this paper new applications to Physics, but focus on finding a proper framework unifying all these structures. Our work seems to be technically complicated enough and applications to Mechanics will be the subject of a separate work.
}

\medskip
The paper is organized as follows. We first recall, in Section 2, needed results on differential
operators and derivative endomorphisms. In Section 3 we investigate, under the name of pseudoalgebras,
algebraic counterparts of algebroids requiring varying differentiability properties for the two entries of the
bracket. The results of Section 4 show that we should relax our traditional understanding of the right anchor
map. A concept of Loday algebroid satisfying all the above requirements is proposed in Definition \ref{d1}
and further detailed in Theorem \ref{LodAld}. In Section 5 we describe a number of new Loday algebroids
containing main canonical examples of Loday brackets on sections of a vector bundle. A natural reduction a Loday pseudoalgebra to a Lie pseudoalgebra is studied in Section 6. For the standard Courant bracket it corresponds to its reduction to the Lie bracket of vector fields.  We then define Loday
algebroid cohomology, Section 7, and interpret in Section 8 our Loday algebroid structures in terms of
homological vector fields of the graded ringed space given by the shuffle multiplication of multidifferential
operators, see Theorem \ref{GeoIntKLA2}. We introduce also the corresponding Cartan calculus.

\section{Differential operators and derivative endomorphisms}
All geometric objects, like manifolds, bundles, maps, sections, etc. will be smooth throughout
this paper.

\begin{definition} A {\it Lie algebroid} structure on a vector bundle $\zt:E\ra M$ is a Lie algebra bracket $[\cdot,\cdot]$ on the real vector space $\cE=\Sec(E)$ of sections of $E$ which satisfies the following compatibility condition related to the $\cA=C^\infty(M)$-module structure in $\cE$:
\be\label{anchor}
\forall \ X,Y\in\cE\ \forall f\in\cA\quad [X,fY]-f[X,Y]=\zr(X)(f)Y\,,
\ee
for some vector bundle morphism $\zr:E\ra\sT M$ covering the identity on $M$ and called the {\it anchor map}.
Here, $\zr(X)=\zr\circ X$ is the vector field on $M$ associated {\it via} $\zr$ with the section $X$.
\end{definition}
Note that the bundle morphism $\zr$ is uniquely determined by the bracket of the Lie algebroid.
What differs a general Lie algebroid bracket from just a Lie module bracket on the $C^\infty(M)$-module
$\Sec(E)$  of sections of $E$ is the fact that it is not $\cA$-bilinear but a certain first-order
bidifferential operator: the adjoint operator $\ad_X=[X,\cdot]$ is a {\it derivative endomorphism}, i.e., the
{\it Leibniz rule}
\be\label{lr}\ad_X(fY)=f\ad_X(Y)+\wh{X}(f)Y
\ee
is satisfied for each $Y\in\cE$ and $f\in\cA$, where $\wh{X}=\zr(X)$ is the vector field on $M$ assigned to
$X$, the {\it anchor} of $X$. Moreover, the assignment $X\mapsto\wh{X}$ is a differential operator of order 0,
as it comes from a bundle map $\zr:E\mapsto \sT M$.

Derivative endomorphisms (also called {\it quasi-derivations}), like differential operators in general, can be
defined for any module $\cE$ over an associative commutative ring $\cA$. Also an extension to superalgebras is
straightforward. These natural ideas go back to Grothendieck and Vinogradov \cite{Vi}.
On the module $\cE$ we have namely a distinguished family $\cA_\cE=\{ f_\cE:f\in\cA\}$ of linear operators
provided by the module structure: $f_\cE(Y)=fY$.
\begin{definition} Let $\cE_i$, $i=1,2$, be modules over the same ring $\cA$. We say that an additive operator $D:\cE_1\ra\cE_2$ is a {\it differential operator of order 0}, if it intertwines $f_{\cE_1}$ with $f_{\cE_2}$, i.e.
\be\label{cm}\zd(f)(D):=D\circ f_{\cE_1}-f_{\cE_2}\circ D\,, \ee
vanishes for all $f\in\cA$. Inductively, we say that $D$ is a {\it differential operator of order} $\le k+1$, if the
commutators (\ref{cm}) are differential operators of order $\le k$. In other words, $D$ is a differential
operator of order $\le k$ if and only if \be\label{do} \forall \ f_1,\dots,f_{k+1}\in\cA\quad
\zd({f_1})\zd({f_2})\cdots\zd({f_{k+1}})(D)=0\,. \ee The corresponding set of differential operators of order
$\le k$ will be denoted by $\cD_k(\cE_1;\cE_2)$ (shortly, $\cD_k(\cE)$, if $\cE_1=\cE_2=\cE$) and the set of
differential operators of arbitrary order (filtered by $\left(\cD_k(\cE_1;\cE_2)\right)_{k=0}^\infty$) by
$\cD(\cE_1;\cE_2)$ (resp., $\cD(\cE)$). We will say that $D$ {\it is of order $k$} if it is of order $\le k$
and not of order $\le k-1$.
\end{definition}
In particular, $\cD_0(\cE_1;\cE_2)=\Hom_\cA(\cE_1;\cE_2)$ is made up by module homomorphisms. Note that in the
case when $\cE_i=\Sec(E_i)$ is the module of sections of a vector bundle $E_i$, $i=1,2$, the concept of
differential operators defined above coincides with the standard understanding. As this will be our standard
geometric model, to reduce algebraic complexity we will assume that $\cA$ is an associative commutative
algebra with unity $1$ over a field $\ka$ of characteristic 0 and all the $\cA$-modules are faithful. In this
case, $\cD(\cE_1;\cE_2)$ is a (canonically filtered) vector space over $\ka$ and, since we work with fields of
characteristic 0, condition (\ref{do}) is equivalent to a simpler condition (see \cite{G}) \be\label{do1}
\forall \ f\in\cA\quad \zd({f})^{k+1}(D)=0\,. \ee If $\cE_1=\cE_2=\cE$, then $\zd(f)(D)=[D,f_\cE]_c$, where
$[\cdot,\cdot]_c$ is the commutator bracket, and elements of $\cA_\cE$ are particular 0-order operators.
Therefore, we can canonically identify $\cA$ with the subspace $\cA_\cE$ in $\cD_0(\cE)$ and use it to
distinguish a particular set of first-order differential operators on $\cE$ as follows.
\begin{definition} {\it Derivative endomorphisms} (or {\it quasi-derivations}) $D:\cE\ra\cE$ are particular first-order differential operators distinguished by the condition
\be\label{anchor1}\forall \ f\in\cA\quad \exists\ \wh{f}\in\cA\quad [D,f_\cE]_c=\wh{f}_\cE\,.\ee
\end{definition}
Since the commutator bracket satisfies the Jacobi identity, one can immediately conclude that
$\wh{f}_\cE=\wh{D}(f)_\cE$ which holds for some derivation $\wh{D}\in\Der(\cA)$ and an arbitrary $f\in\cA$
\cite{G1}. Derivative endomorphisms form a submodule $\Der(\cE)$ in the $\cA$-module $\End_{\mathbb{K}}(\cE)$
of $\mathbb{K}$-linear endomorphisms of $\cE$ which is simultaneously a Lie subalgebra over $\ka$ with respect
to the commutator bracket. The linear map,
$$\Der(\cE)\ni D\mapsto\wh{D}\in\Der(\cA)\,,$$
called the {\it universal anchor map}, is a differential operator of order 0, $\wh{fD}=f\wh{D}$. The Jacobi
identity for the commutator bracket easily implies (see \cite[Theorem 2]{G1})
\be\label{anchor2} [D_1,D_2]_c^{\wh{}}=[\wh{D_1},\wh{D_2}]_c\,.
\ee
It is worth remarking (see \cite{G1}) that also $\cD(\cE)$ is a Lie subalgebra in $\End_\ka(\cE)$, as
\be\label{qP}[\cD_k(\cE),\cD_l(\cE)]_c\subset\cD_{k+l-1}(\cE)\,,
\ee
and an associative subalgebra, as
\be\label{qP1}\cD_k(\cE)\circ\cD_l(\cE)\subset\cD_{k+l}(\cE)\,,
\ee
that makes $\cD(\cE)$ into a canonical example of a {\it quantum Poisson algebra} in the terminology of
\cite{GP}.

It was pointed out in \cite{KM} that the concept of derivative endomorphism can be traced back to N.~Jacobson
\cite{Ja1,Ja2} as a special case of his {\it pseudo-linear endomorphism}. It has appeared also in \cite{Ne}
under the name {\it module derivation} and was used to define linear connections in the algebraic setting. In
the geometric setting of Lie algebroids it has been studied in \cite{Mac} under the name {\it covariant
differential operator}. For more detailed history and recent development we refer to \cite{KM}.

Algebraic operations in differential geometry have usually a local character in order to be treatable with
geometric methods. On the pure algebraic level we should work with differential (or multidifferential)
operations, as tells us the celebrated Peetre Theorem \cite{Pe,Pe1}. The algebraic concept of a
multidifferential operator is obvious. For a $\ka$-multilinear operator $D:\cE_1\ti\cdots\ti\cE_p\ra\cE$ and
each $i=1,\dots,p$, we say that  $D$ is a {\it differential operator of order $\le k$ with respect to the
$i$th variable}, if, for all $y_j\in\cE_j$, $j\ne i$,
$$D(y_1,\dots,y_{i-1},\,\cdot\,,y_{i+1},\dots, y_p):\cE_i\ra\cE$$
is a differential operator of order $\le k$. In other words,
\be\label{mdo}
\forall \ f\in\cA\quad \zd_i({f})^{k+1}(D)=0\,,
\ee
where
\be
\zd_i(f)D(y_1,\dots,y_p)=D(y_1,\dots,fy_i,\dots,y_p)-fD(y_1,\dots,y_p)\,.
\ee
Note that the operations $\zd_i(f)$ and $\zd_j(g)$ commute. We say that the operator $D$ {\it is a
multidifferential operator of order $\le n$}, if it is of order $\le n$ with respect to each variable
separately. This means that, fixing any $p-1$ arguments, we get a differential operator of order $\le n$. A
similar, but stronger, definition is the following
\begin{definition} We say that a multilinear operator $D:\cE_1\ti\cdots\ti\cE_p\ra\cE$ is a {\it multidifferential operator of total order $\le k$}, if
\be\label{mdo1}
\forall \ f_1,\dots,f_{k+1}\in\cA\ \forall\ i_1,\dots,i_{k+1}=1,\dots,p\quad
\left[\zd_{i_1}({f_1})\zd_{i_2}({f_2})\cdots\zd_{i_{k+1}}({f_{k+1}})(D)=0\right]\,.
\ee
\end{definition}
\noindent Of course, a multidifferential operator of total order $\le k$ is a multidifferential operator of
order $\le k$. It is also easy to see that a $p$-linear differential operator of order $\le k$ is a
multidifferential operator of total order $\le pk$. In particular, the Lie bracket of vector fields (in fact,
any Lie algebroid bracket) is a bilinear differential operator of total order $\le 1$.

\section{Pseudoalgebras}
Let us start this section with recalling that Loday, while studying relations between Hochschild and cyclic
homology in the search for obstructions to the periodicity of algebraic K-theory, discovered that one can skip
the skew-symmetry assumption in the definition of Lie algebra, still having a possibility to define an
appropriate (co)homology (see \cite{Lo1,LP} and \cite[Chapter 10.6]{Lo}). His Jacobi identity for such
structures was formally the same as the classical Jacobi identity in the form
\be\label{JI} [x,[y,z]]=[[x,y],z]+[y,[x,z]]. \ee This time,
however, this is no longer equivalent to \be\label{JI1} [[x,y],z]=[[x,z],y]+[x,[y,z]], \ee nor to
\be\label{JI2} [x,[y,z]]+[y,[z,x]]+[z,[x,y]]=0, \ee since we have no skew-symmetry. Loday called such
structures {\it Leibniz algebras}, but to avoid collision with another concept of {\it Leibniz brackets} in
the literature, we shall call them {\it Loday algebras}. This is in accordance with the terminology of {\cite{K-S}},
where analogous structures in the graded case are defined. Note that the identities (\ref{JI}) and (\ref{JI1})
have an advantage over the identity (\ref{JI2}) obtained by cyclic permutations, since they describe the
algebraic facts that the left-regular (resp., right-regular) actions are left (resp., right) derivations. This
was the reason to name the structure `Leibniz algebra'.

Of course, there is no particular reason not to define Loday algebras by means of (\ref{JI1}) instead of
(\ref{JI}) (and in fact, it was the original definition by Loday), but this is not a substantial difference,
as both categories are equivalent via transposition of arguments. We will use the form (\ref{JI}) of the
Jacobi identity.

Our aim is to find a proper generalization of the concept of Loday algebra in a way similar to that in which
Lie algebroids generalize Lie algebras. If one thinks about a generalization of a concept of Lie algebroid as
operations on sections of a vector bundle including operations (brackets) which are non-antisymmetric or which
do not satisfy the Jacobi identity, and are not just $\cA$-bilinear, then it is reasonable, on one hand, to
assume differentiability properties of the bracket as close to the corresponding properties of Lie algebroids
as possible and, on the other hand, including all known natural examples of such brackets. This is not an
easy task, since, as we will see soon, some natural possibilities provide only few new examples.

To present a list of these possibilities, we propose the following definitions serving in the pure algebraic
setting.
\begin{definition}\label{def} Let $\cE$ be a faithful module over an associative commutative algebra $\cA$ over a field $\ka$ of characteristic 0. A a $\ka$-bilinear bracket $B=[\cdot,\cdot] :\cE\ti\cE\ra\cE$ on the module $\cE$
\begin{enumerate}
\item is called a {\it {faint} pseudoalgebra bracket}, if  $B$ is a bidifferential operator; \item is called
a {\it weak pseudoalgebra bracket}, if $B$ is a bidifferential operator of degree $\le 1$; \item is called a
{\it quasi pseudoalgebra bracket}, if $B$ is a bidifferential operator of total degree $\le 1$; \item is
called a {\it pseudoalgebra bracket}, if $B$ is a bidifferential operator of total degree $\le 1$ and the {\it
adjoint map} $\ad_X=[X,\cdot]:\cE\ra\cE$ is a derivative endomorphism for each $X\in\cE$; \item is called a
{\it QD-pseudoalgebra bracket}, if the {\it adjoint maps} $\ad_X,\ad_X^r:\cE\ra\cE$, \be\label{ad}
\ad_X=[X,\cdot]\,,\quad \ad_X^r=[\cdot,X]\,\quad(X\in\cE)\,,\ee associated with $B$ are derivative
endomorphisms (quasi-derivations); \item is called a {\it strong pseudoalgebra bracket}, if $B$ is a
bidifferential operator of total degree $\le 1$ and the {\it adjoint maps} $\ad_X,\ad_X^r:\cE\ra\cE$,
\be\label{ad1} \ad_X=[X,\cdot]\,,\quad \ad_X^r=[\cdot,X]\,\quad(X\in\cE)\,, \ee are derivative endomorphisms.
\end{enumerate}
We call the module $\cE$ equipped with such a bracket, respectively, a {\it {faint} pseudoalgebra}, {\it weak
pseudoalgebra} etc. If the bracket is symmetric (skew-symmetric), we speak about {faint}, weak, etc., {\it
symmetric} ({\it skew}) {\it pseudoalgebras}. If the bracket satisfies the Jacobi identity (\ref{JI}), we
speak about local, weak, etc.,  {\it Loday pseudoalgebras}, and if the bracket is a Lie algebra bracket, we
speak about local, weak, etc., {\it Lie pseudoalgebras}. If $\cE$ is the $\cA=C^\infty(M)$ module of sections
of a vector bundle $\zt:E\ra M$, we refer to the above pseudoalgebra structures as to {\it algebroids}.
\end{definition}
\begin{thm} If \ $[\cdot,\cdot]$ is a pseudoalgebra bracket, then the map
$$\zr:\cE\ra \Der(\cA)\,,\quad \zr(X)=\wh{\ad_X}\,,$$
called the {\em anchor map}, is $\cA$-linear, $\zr(fX)=f\zr(X)$, and
\be\label{aanchor}[X,fY]=f[X,Y]+\zr(X)(f)Y
\ee
for all $X,Y\in\cE$, $f\in\cA$. Moreover, if \ $[\cdot,\cdot]$ satisfies additionally the Jacobi identity,
i.e., we deal with a Loday pseudoalgebra, then the anchor map is a homomorphism into the commutator bracket,
\be\label{anhom}
\zr\left([X,Y]\right)=[\zr(X),\zr(Y)]_c\,.
\ee
\end{thm}
\begin{proof}
Since the bracket $B$ is a bidifferential operator of total degree $\le 1$, we have $\zd_1(f)\zd_2(g)B=0$ for
all $f,g\in\cA$. On the other hand, as easily seen,
\be\label{QD}(\zd_1(f)\zd_2(g)B)(X,Y)=\left(\zr(fX)-f\zr(X)\right)(g)Y\,,\ee
and the module is faithful, it follows $\zr(fX)=f\zr(X)$. The identity (\ref{anhom}) is a direct implication
of the Jacobi identity combined with (\ref{aanchor}).

\end{proof}
\begin{thm} If \ $[\cdot,\cdot]$ is a QD-pseudoalgebra bracket, then it is a weak pseudoalgebra bracket and admits two {\it anchor maps}
$$\zr,\zr^r:\cE\ra \Der(\cA)\,,\quad \zr(X)=\wh{\ad_X}\,,\ \zr^r=-\wh{\ad^r}\,,$$
for which we have
\be\label{anchors}[X,fY]=f[X,Y]+\zr(X)(f)Y\,,\quad [fX,Y]=f[X,Y]-\zr^r(X)(f)Y\,,
\ee
for all $X,Y\in\cE$, $f\in\cA$. If the bracket is skew-symmetric, then both anchors coincide, and if the
bracket is a strong QD-pseudoalgebra bracket, they are $\cA$-linear. Moreover, if \ $[\cdot,\cdot]$ satisfies
additionally the Jacobi identity, i.e., we deal with a Loday QD-pseudoalgebra, then, for all $X,Y\in\cE$,
\be\label{anhomn}
\zr\left([X,Y]\right)=[\zr(X),\zr(Y)]_c\,.
\ee
\end{thm}
\begin{proof}
Similarly as above,
$$(\zd_2(f)\zd_2(g)B)(X,Y)=\zr(X)(g)fY-f\zr(X)(g)Y=0\,,$$
so $B$ is a first-order differential operator with respect to the second argument. The same can be done for
the first argument.

Next, as for any QD-pseudoalgebra bracket $B$ we have, analogously to (\ref{QD}),
\be\label{QD1}(\zd_1(f)\zd_2(g)B)(X,Y)=\left(\zr(fX)-f\zr(X)\right)(g)Y=\left(\zr^r(gY)-g\zr^r(Y)\right)(f)X\,,\ee
both anchor maps are $\cA$-linear if and only if $D$ is of total order $\le 1$. The rest follows analogously
to the previous theorem.

\end{proof}

The next observation is that quasi pseudoalgebra structures on an $\cA$-module $\cE$ have certain analogs of
anchor maps, namely $\cA$-module homomorphisms $b=b^r,b^l:\cE\ra\Der(\cA)\otimes_\cA\End(\cE)$. For every
$X\in\cE$ we will view $b(X)$ as an $\cA$-module homomorphism $b(X):\zW^1\ot_\cA\cE\ra\cE$, where $\zW^1$
is the $\cA$-submodule of $\Hom_\cA(\Der(\cA);\cA)$ generated by $\xd\cA=\{\xd f:f\in\cA\}$ and $\la\xd
f,D\ran=D(f)$. Elements of $\,\Der(\cA)\otimes_\cA\End(\cE)$ act on elements of $\zW^1\otimes_\cA \cE$ in the
obvious way: $(V\otimes\Phi)(\zw\otimes X)=\la V,\zw\ran\Phi(X)$.

\begin{thm}\label{T1} A $\,\ka$-bilinear bracket $B=[\cdot,\cdot]$ on an $\cA$-module $\cE$ defines a quasi pseudoalgebra structure if
and only if there are $\cA$-module homomorphisms \be\label{Lqa} b^r,b^l:\cE\ra\Der(\cA)\otimes_\cA\End(\cE)\,,
\ee
called {\em generalized anchor maps, right and left}, such that, for all $X,Y\in\cE$ and all $f\in\cA$,
\be\label{aLqa} [X,fY]=f[X,Y]+b^l(X)(\xd f\otimes Y)\,,\quad [fX,Y]=f[X,Y]-b^r(Y)(\xd f\otimes X)\,. \ee The
generalized anchor maps are actual anchor maps if they take values in $\Der(\cA)\otimes_\cA\{\ope{Id}_\cE\}$.
\end{thm}
\begin{proof} Assume first that the bracket $B$ is a bidifferential operator of total degree $\le 1$ and define a
three-linear map of vector spaces  $A:\cE\ti\cA\ti\cE\ra\cE$ by
$$A(X,g,Y)=(\zd_2(g)B)(X,Y)=[X,gY]-g[X,Y]\,.
$$
It is easy to see that $A$ is $\cA$-linear with respect to the first and the third argument, and a derivation
with respect to the second. Indeed, as
$$
(\zd_1(f)\zd_2(g)B)(X,Y)=A(fX,g,Y)-fA(X,g,Y)=0\,,
$$
we get $\cA$-linearity with respect to the first argument. Similarly, from $\zd_2(f)\zd_2(g)B=0$, we get the
same conclusion for the third argument. We have also
\bea\nn A(X,fg,Y)&=&[X,fgY]-fg[X,Y]=[X,fgY]-f[X,gY]+f[X,gY]-fg[X,Y]\\
&=&A(X,f,gY)+fA(X,g,Y)=gA(X,f,Y)+fA(X,g,Y)\,,\label{w1} \eea thus the derivation property. This implies that
$A$ is represented by an $\cA$-module homomorphism $b^l:\cE\ra\Der(\cA)\otimes_\cA\End(\cE)$. Analogous
considerations give us the right generalized anchor map $b^r$.

Conversely, assume the existence of both generalized anchor maps. Then, the map $A$ defined as above reads
$A(X,f,Y)=b^l(X)(\xd f\ot Y)$, so is $\cA$-linear with respect to $X$ and $Y$. Hence,
$$(\zd_1(f)\zd_2(g)B)(X,Y)=A(fX,g,Y)-fA(X,g,Y)=0$$ and
$$(\zd_2(f)\zd_2(g)B)(X,Y)=A(X,g,fY)-fA(X,g,Y)=0\,.$$
A similar reasoning for $b^r$ gives $(\zd_1(f)\zd_1(g)B)(X,Y)=0$, so the bracket is a bidifferential operator
of total order $\le 1$.
\end{proof}

In the case when we deal with a quasi algebroid, i.e., $\cA=\cim$ and $\cE=\Sec(E)$ for a vector bundle
$\zt:E\ra M$, the generalized anchor maps (\ref{Lqa}) are associated with vector bundle maps that we denote
(with some abuse of notations) also by $b^r,b^l$,
$$b^r,b^l:E\ra\sT M\otimes_M\End(E)\,,$$
covering the identity on $M$. Here, $\End(E)$ is the endomorphism bundle of $E$, so $\End(E)\simeq E^\ast\ot_M
E$. The induced maps of sections produce from sections of $E$ sections of $\sT M\otimes_M\End(E)$ which, in
turn, act on sections of $\sT^\ast M\otimes_M E$ in the obvious way. An algebroid version of Theorem \ref{T1}
is the following.

\begin{thm}\label{T1a} An $\R$-bilinear bracket $B=[\cdot,\cdot]$ on the real space $\Sec(E)$ of sections of a vector
bundle $\zt:E\ra M$ defines a quasi
algebroid structure if and only if there are vector bundle morphisms
\be\label{Lqa1} b^r,b^l:E\ra\sT
M\otimes_M\End(E) \ee covering the identity on $M$, called {\em generalized anchor maps, right and left}, such
that, for all $X,Y\in\Sec(E)$ and all $f\in C^\infty(M)$, (\ref{aLqa}) is satisfied. The generalized anchor maps are
actual anchor maps, if they take values in $\sT M\otimes\la \ope{Id}_E\ran\simeq\sT M$.
\end{thm}

\section{Loday algebroids}
Let us isolate and specify the most important particular cases of Definition \ref{def}.
\begin{definition}\

\begin{enumerate}
\item A {\it {faint} Loday algebroid} (resp., {\it {faint} Lie algebroid}) on a vector bundle $E$ over a
base manifold $M$ is a Loday bracket (resp., a Lie bracket) on the $C^\infty(M)$-module $\Sec(E)$ of smooth
sections of $E$ which is a bidifferential operator.

\item A {\it weak Loday algebroid} (resp., {\it weak Lie algebroid}) on a vector bundle $E$ over a base
manifold $M$ is a Loday bracket (resp., a Lie bracket) on the $C^\infty(M)$-module $\Sec(E)$ of smooth
sections of $E$ which is a bidifferential operator of degree $\le 1$ with respect to each variable separately.

\item A {\it Loday quasi algebroid} (resp., {\it Lie quasi algebroid}) on a vector bundle $E$ over a base
manifold $M$ is a Loday bracket (resp., Lie bracket) on the $C^\infty(M)$-module $\Sec(E)$ of smooth sections
of $E$ which is a bidifferential operator of total degree $\le 1$.

\item A {\it QD-algebroid} (resp., {\it skew QD-algebroid, Loday QD-algebroid, Lie QD-algebroid}) on a vector
bundle $E$ over a base manifold $M$ is an $\R$-bilinear bracket (resp., skew bracket, Loday bracket, Lie
bracket) on the $C^\infty(M)$-module $\Sec(E)$ of smooth sections of $E$ for which the adjoint operators
$\ad_X$ and $\ad_X^r$ are derivative endomorphisms.
\end{enumerate}
\end{definition}
\begin{remark} {\it Lie pseudoalgebras} appeared first in the paper of Herz  \cite{He}, but one can find similar concepts
under more than a dozen of names in the literature  (e.g. {\it Lie modules,  $(R,A)$-Lie   algebras,
Lie-Cartan pairs, Lie-Rinehart algebras, differential algebras}, etc.). Lie algebroids were introduced  by
Pradines \cite{Pr} as infinitesimal parts of differentiable groupoids. In the same year a book by Nelson was
published where a general theory of Lie modules, together with a big part of the corresponding differential
calculus, can be found. We also refer to a survey article by Mackenzie \cite{Ma}. QD-algebroids, as well as
Loday QD-algebroids  and Lie QD-algebroids, have been introduced in \cite{G1}. In \cite{GU,GGU} Loday strong
QD-algebroids have been called Loday algebroids and strong QD-algebroids have been called just {\it
algebroids}. The latter served as geometric framework for generalized Lagrange and Hamilton formalisms.

In the case of line bundles, $\rk E=1$, Lie QD-algebroids are exactly {\it local Lie algebras} in the sense of
Kirillov \cite{Ki}. They are just {\it Jacobi brackets}, if the bundle is trivial, $\Sec(E)=\cim$. Of course,
Lie QD-algebroid brackets are first-order bidifferential operators by definition, while Kirillov has
originally started with considering Lie brackets on sections of line bundles determined by local operators and
has only later discovered that these operators have to be bidifferential operators of first order. A purely
algebraic version of Kirillov's result has been proven in \cite{G}, Theorems 4.2 and 4.4, where bidifferential
Lie brackets on associative commutative algebras containing no nilpotents have been considered.
\end{remark}

\begin{example}\label{e1} Let us consider a Loday algebroid bracket in the sense of \cite{Ha,HM,ILMP,KS,MM,Wa}, i.e., a Loday algebra bracket $[\cdot,\cdot]$ on the $\cim$-module $\cE=\Sec(E)$ of sections of a vector bundle $\zt:E\ra M$ for which there is a vector bundle morphism $\zr:E\ra\sT M$ covering the identity on $M$ (the left anchor map) such that (\ref{anchor}) is satisfied. Since, due to (\ref{anhom}), the anchor map is necessarily a homomorphism of the Loday bracket into the Lie bracket of vector fields, our Loday algebroid is just a Lie algebroid in the case when $\zr$ is injective.
In the other cases the anchor map does not determine the Loday algebroid structure, in particular does not
imply any locality of the bracket with respect to the first argument. Thus, this concept of Loday algebroid is
not geometric.

For instance, let us consider a Whitney sum bundle $E=E_1\op_M E_2$ with the canonical projections $p_i:E\ra
E_i$ and any $\R$-linear map $\zf:\Sec(E_1)\ra\cim$. Being only $\R$-linear, $\zf$ can be chosen very strange
non-geometric and non-local. Define now the following bracket on $\Sec(E)$:
$$[X,Y]=\zf(p_1(X))\cdot p_2(Y)\,.$$
It is easy to see that this is a Loday bracket which admits the trivial left anchor, but the bracket is
non-local and non-geometric as well.
\end{example}

\begin{example} A standard example of a weak Lie algebroid bracket is a Poisson (or, more generally, Jacobi) bracket $\{\cdot,\cdot\}$ on $C^\infty(M)$ viewed as a $\cim$-module of section of the trivial line bundle $M\ti\R$.
It is a bidifferential operator of order $\le 1$ and the total order $\le 2$. It is actually a Lie QD-algebroid
bracket, as $\ad_f$ and $\ad_f^r$ are, by definition, derivations (more generally, first-order differential
operators). Both anchor maps coincide and give the corresponding Hamiltonian vector fields, $\zr(f)(g)=\{
f,g\}$. The map $f\mapsto \zr(f)$ is again a differential operator of order 1, so is not implemented by a
vector bundle morphism $\zr:M\ti\R\ra\sT M$. Therefore, this weak Lie algebroid is not a Lie algebroid. {This has a straightforward generalization to {\em Kirillov brackets} being local Lie brackets on sections of a line bundle \cite{Ki}.}
\end{example}
\begin{example}
Various brackets are associated with a volume form $\zw$ on a manifold $M$ of dimension $n$ (see e.g. \cite{Li}). Denote with $\cX^k(M)$ (resp., $\zW^k(M)$) the spaces of $k$-vector fields (resp., $k$-forms) on $M$. As
the contraction maps $\cX^k(M)\ni K\mapsto i_K\zw\in\zW^{n-k}(M)$ are isomorphisms of $\cim$-modules, to the
de Rham cohomology operator $\xd:\zW^{n-k-1}(M)\ra\zW^{n-k}(M)$ corresponds a homology operator
$\zd:\cX^k(M)\ra\cX^{k-1}(M)$. The skew-symmetric bracket $B$ on $\cX^2(M)$ defined in \cite{Li} by
$B(t,u)=-\zd(t)\we\zd(u)$ is not a Lie bracket, since its Jacobiator $B(B(t,u),v)+c.p.$ equals
$\zd(\zd(t)\we\zd(u)\we\zd(v))$. A solution proposed in \cite{Li} depends on considering the algebra $N$ of bivector
fields modulo $\zd$-exact bivector fields for which the Jacobi anomaly disappears, so that $N$ is a Lie
algebra.

Another option is to resign from skew-symmetry and define the corresponding {faint} Loday algebroid. In view
of the duality between $\cX^2(M)$ and $\zW^{n-2}$, it is possible to work with $\zW^{n-2}(M)$ instead. For
$\zg\in\zW^{n-2}(M)$ we define the vector field $\wh{\zg}\in\cX(M)$ from the formula $i_{\wh{\zg}}\zw=\xd\zg$.
The bracket in $\zW^{n-2}(M)$ is now defined by {(see \cite{Lo})}
$$\{\zg,\zb\}_\zw=\Ll_{\wh{\zg}}\zb=i_{\wh{\zg}}i_{\wh{\zb}}\zw+\xd
i_{\wh{\zg}}\zb\,.$$ Since we have
$$i_{[\wh{\zg},\wh{\zb}]_{vf}}\zw=\Ll_{\wh{\zg}}i_{\wh{\zb}}\zw-i_{\wh{\zb}}\Ll_{\wh{\zg}}\zw=
\xd i_{\wh{\zg}}i_{\wh{\zb}}\zw=\xd \{\zg,\zb\}_\zw\,,$$ it holds
$$\{\zg,\zb\}_\zw^{\wh{}}=[\wh{\zg},\wh{\zb}]_{vf}\,.$$
Therefore,
$$\{\{\zg,\zb\}_\zw,\zh\}_\zw=\Ll_{\{\zg,\zb\}_\zw^{\wh{}}}\zh=
\Ll_{\wh{\zg}}\Ll_{\wh{\zb}}\zh-\Ll_{\wh{\zb}}\Ll_{\wh{\zg}}\zh=\{\zg,\{\zb,\zh\}_\zw\}_\zw
-\{\zb,\{\zg,\zh\}_\zw\}_\zw\,,$$ so the Jacobi identity is satisfied and we deal with a Loday algebra. This
is in fact a {faint} Loday algebroid structure on $\we^{n-2}\sT^\ast M$ with the left anchor
$\zr(\zg)=\wh{\zg}$. This bracket is a bidifferential operator which is first-order with respect to the second
argument and second-order with respect to the first one.
\end{example}

Note that Lie QD-algebroids are automatically Lie algebroids, if the rank of the bundle $E$ is $>1$
\cite[Theorem 3]{G1}. Also some other of the above concepts do not produce qualitatively new examples.
\begin{thm} (\cite{G1,GM,GM1})\

\begin{description}
\item{(a)} Any Loday bracket on $\cim$ (more generally, on sections of a line bundle) which is a
bidifferential operator is actually a Jacobi bracket (first-order and skew-symmetric). \item{(b)} Let
$[\cdot,\cdot]$ be a Loday bracket on sections of a vector bundle $\zt:E\ra M$, admitting anchor maps
$\zr,\zr^r:\Sec(E)\ra\cX(M)$ which assign vector fields to sections of $E$ and such that (\ref{anchors}) is
satisfied (Loday QD-algebroid on $E$). Then, the anchors coincide, $\zr=\zr^r$, and the bracket is
skew-symmetric at points $p\in M$ in the support of $\zr=\zr^r$. Moreover, if the rank of $E$ is $>1$, then
the anchor maps are $\cim$-linear, i.e. they come from a vector bundle morphism $\zr=\zr^r:E\ra M$. In other words,
any Loday QD-algebroid is actually, around points where one anchor does not vanish, a Jacobi bracket if
$\rk(E)=1$, or Lie algebroid bracket if $\rk(E)>1$.
\end{description}
\end{thm}
The above results show that relaxing skew-symmetry and considering Loday brackets on $\cim$ or $\Sec(E)$ does
not lead to new structures (except for just bundles of Loday algebras), if we assume differentiability in the
first case and the existence of both (possibly different) anchor maps in the second. Therefore, a definition
of Loday algebroids that admits a rich family of new examples, must resign from the traditionally understood
right anchor map.

\medskip
\noindent The definition of the main object of our studies can be formulated as follows.
\begin{definition}\label{d1}
A {\it Loday algebroid} on a vector bundle $E$ over a base manifold $M$ is a Loday bracket on the
$C^\infty(M)$-module $\Sec(E)$ of smooth sections of $E$ which is a bidifferential operator of total degree
$\le 1$ and for which the adjoint operator $\ad_X$ is a derivative endomorphism.
\end{definition}
{Of course, the above definition of Loday algebroid is stronger than those known in the literature (e.g. \cite{Ha,HM,ILMP,KS,MM,Wa}), which assume only the existence of a left anchor and put no differentiability requirements for the first variable.
}
\begin{thm}\label{LodAld} A Loday bracket $[\cdot,\cdot]$ on the real space $\Sec(E)$ of sections of a vector bundle $\zt:E\ra M$
defines a Loday algebroid structure
if and only if there are vector bundle morphisms
\be\label{Lqa1a} \zr:E\ra\sT M\,, \quad \za:E\ra\sT
M\otimes_M\End(E)\,, \ee covering the identity on $M$, such that, for all $X,Y\in\Sec(E)$ and all $f\in
C^\infty(M)$, \be\label{aLqa1} [X,fY]=f[X,Y]+\zr(X)(f)Y\,,\quad [fX,Y]=f[X,Y]-\zr(Y)(f)X+\za(Y)(\xd f\otimes
X)\,. \ee If this is the case, {the anchors are uniquely determined and} the left anchor induces a homomorphism of the Loday bracket into the bracket
$[\cdot,\cdot]_{vf}$ of vector fields,
$$\zr([X,Y])=[\zr(X),\zr(Y)]_{vf}\,.$$
\end{thm}
\begin{proof} This is a direct consequence of Theorem \ref{T1a} and the fact that an algebroid bracket has the left anchor map. We
just write the generalized right anchor map as $b^r=\zr\ot I-\za$.
\end{proof}

To give a local form of a Loday algebroid bracket, let us recall that sections $X$ of the vector bundle $E$
can be identified with linear (along fibers) functions $\zi_X$ on the dual bundle $E^\ast$. Thus, fixing local
coordinates $(x^a)$ in $M$ and a basis of local sections $e_i$ of $E$, we have a corresponding system
$(x^a,\zx_i=\zi_{e_i})$ of affine coordinates in $E^\ast$. As local sections of $E$ are identified with linear
functions $\zs=\zs^i(x)\zx_i$, the Loday bracket is represented by a bidifferential operator $B$ of total
order $\le 1$:
$$B(\zs^i_1(x)\zx_i,\zs^j_2(x)\zx_j)=c_{ij}^k(x)\zs^i_1(x)\zs^j_2(x)\zx_k+\zb_{ij}^{ak}(x)\frac{\pa\zs^i_1}
{\pa x^a}(x)\zs^j_2(x)\zx_k+\zg_{ij}^{ak}(x)\zs^i_1(x)\frac{\pa\zs^j_2}{\pa x^a}(x)\zx_k\,.$$
Taking into account the existence of the left anchor, we have
\bea\label{loc} B(\zs^i_1(x)\zx_i,\zs^j_2(x)\zx_j)&=&c_{ij}^k(x)\zs^i_1(x)\zs^j_2(x)\zx_k+\za_{ij}^{ak}(x)
\frac{\pa\zs^i_1}{\pa x^a}(x)\zs^j_2(x)\zx_k\\
&&+\zr_{i}^{a}(x)\left(\zs^i_1(x)\frac{\pa\zs^j_2}{\pa x^a}(x)-\frac{\pa\zs^j_1}{\pa
x^a}(x)\zs^i_2(x)\right)\zx_j\,. \nn\eea Since sections of $\End(E)$ can be written in the form of linear
differential operators, we can rewrite (\ref{loc}) in the form
\be\label {loc1}
B=c_{ij}^k(x)\zx_k\pa_{\zx_i}\ot\pa_{\zx_j}+\za_{ij}^{ak}(x)\zx_k\pa_{x^a}\pa_{\zx_i}\ot\pa_{\zx_j}
+\zr_{i}^{a}(x)\pa_{\zx_i}\we\pa_{x^a}\,.
\ee
Of course, there are additional relations between coefficients of $B$ due to the fact that the Jacobi identity
is satisfied.

\section{Examples}
\subsection{Leibniz algebra}

Of course, a finite-dimensional Leibniz algebra is a Leibniz algebroid over a point.

\subsection{Courant-Dorfman bracket}

The {\it Courant bracket} is defined on sections of $\cT M=\sT M\oplus_M\sT^*M$ as follows:
\be\label{CB}[X+\zw,Y+\zh]=[X,Y]_{vf}+\Ll_X\zh-\Ll_Y\zw-\frac{1}{2}\left(d\,i_X\zh-d\,i_Y\zw\right)\,.
\ee This bracket is antisymmetric, but it does not satisfy the
Jacobi identity; the Jacobiator is an exact 1-form. It is, as easily seen, given by a bidifferential operator
of total order $\le 1$, so it is a skew quasi algebroid.

\medskip
The {\it Dorfman bracket} is defined on the same module of sections. Its definition is the same as for
Courant, except that the corrections and the exact part of the second Lie derivative disappear:
\be\label{CD}[X+\zw,Y+\zh]=[X,Y]_{vf}+\Ll_X\zh-i_Y\,\xd\zw=[X,Y]_{vf}+i_X\,\xd\zh-i_Y\,\xd\zw+\xd\,i_X\zh\,.
\ee
This bracket is visibly non skew-symmetric, but it is a Loday bracket which is bidifferential of total order
$\le 1$. Moreover, the Dorfman bracket admits the classical left anchor map
\be\label{CDa}\zr:\cT M=\sT M\oplus_M\sT^*M\ra \sT M
\ee
which is the projection onto the first component. Indeed,
$$[X+\zw,f(Y+\zh)]=[X,fY]_{vf}+\Ll_Xf\zh-i_{fY}\,\xd\zw=f[X+\zw,Y+\zh]+X(f)(Y+\zh)\,.$$
For the right generalized anchor we have
\beas[f(X+\zw),Y+\zh]&=&[fX,Y]_{vf}+i_{fX}\,\xd\zh-i_Y\,\xd(f\zw)+\xd\,i_{fX}\zh\\
&=&f[X+\zw,Y+\zh]-Y(f)(X+\zw)+\xd f\we (i_X\zh+i_Y\zw)\,,\eeas so that
 $$\za(Y+\zh)(\xd f\ot(X+\zw))=\xd f\we (i_X\zh+i_Y\zw)=2\la X+\zw,Y+\zh\ran_+\cdot\xd f\,,$$
where
$$\la
X+\zw,Y+\zh\ran_+=\frac{1}{2}\left(i_X\zh+i_Y\zw\right)=\frac{1}{2}\left(\la X,\zh\ran+\la
Y,\zw\ran\right)\,,$$ is a symmetric nondegenerate bilinear form on $\cT M$ (while $\la\cdot,\cdot\ran$ is the
canonical pairing). We will refer to it, though it is not positively defined, as the {\it scalar product} in
the bundle $\cT M$.

Note that $\za(Y+\zh)$ is really a section of $\sT M\ot_M\End(\sT M\oplus_M\sT^\ast M)$ that in local
coordinates reads
$$\za(Y+\zh)=\sum_k\partial_{x^k}\ot(\xd x^k\we (i_{\zh}+i_Y))\,.$$
Hence, the Dorfman bracket is a Loday algebroid bracket.\medskip

It is easily checked that the Courant bracket is the antisymmetrization of the Dorfman bracket, and that the
Dorfman bracket is the Courant bracket plus $\xd\la X+\zw,Y+\zh\ran_+$
\subsection{Twisted Courant-Dorfman bracket}
The Courant-Dorfman bracket can be twisted by adding a term associated with a 3-form $\zY$ \cite{YKS1,SW}:
\be\label{TCD}[X+\zw,Y+\zh]=[X,Y]_{vf}+\Ll_X\zh-i_Y\,\xd\zw+i_{X\we Y}\zY\,.
\ee
It turns out that this bracket is still a Loday bracket if the 3-form $\zY$ is closed. As the added term is
$\cim$-linear with respect to $X$ and $Y$, the anchors remain the same, thus we deal with a Loday algebroid.

\subsection{Courant algebroid}
Courant algebroids -- structures generalizing the Courant-Dorfman bracket on $\cT M$ -- were introduced as as
double objects for Lie bialgebroids by Liu, Weinstein and Xu \cite{LWX} in a bit complicated way. It was shown
by Roytenberg \cite{Roy0} that a Courant algebroid can be equivalently defined as a vector bundle $\zt:E\ra M$
with a Loday bracket on $\Sec(E)$, an anchor $\zr:E\ra \sT M$, and a symmetric nondegenerate inner product
$(\cdot,\cdot)$ on $E$, related by a set of four additional properties. It was further observed \cite{Uch,
GM2} that the number of independent conditions can be reduced.

\begin{definition} A \textit{Courant algebroid} is a vector bundle $\zt:E\ra M$ equipped with a Leibniz
bracket $[\cdot,\cdot]$ on $\Sec(E)$, a vector bundle map (over the identity) $\zr:E\ra \sT M$, and a
nondegenerate symmetric bilinear form (scalar product) $({\cdot}|{\cdot})$ on $E$ satisfying the identities
\bea\label{4} &\zr(X)( Y|Y)=2( X|[Y,Y]),\\ &\zr(X)( Y|Y)=2( [X,Y]|Y).\label{5} \eea
\end{definition}

\medskip\noindent Note that (\ref{4}) is equivalent to
\begin{equation}\label{4a}
\zr(X)( Y|Z)=( X|[Y,Z]+[Z,Y]).
\end{equation}
Similarly, (\ref{5}) easily implies the invariance of the pairing $({\cdot},{\cdot})$ with respect to the
adjoint maps
\begin{equation}\label{6}\zr(X)(Y|Z)=( [X,Y]|Z)+( Y|[X,Z]),
\end{equation}
which in turn shows that $\zr$ is the anchor map for the left multiplication: \be\label{zr}
[X,fY]=f[X,Y]+\zr(X)(f)Y\,. \ee Twisted Courant-Dorfman brackets are examples of Courant algebroid brackets
with $({\cdot},{\cdot})=\la\cdot,\cdot\ran_+$ as the scalar product. Defining a derivation
$\sD:\cim\ra\Sec(E)$ by means of the scalar product
\be\label{D}(\sD(f)|X)=\frac{1}{2}\zr(X)(f)\,, \ee we get out of
(\ref{4a}) that \be\label{4c}[Y,Z]+[Z,Y]=2\sD(Y|Z)\,. \ee This, combined with (\ref{zr}), implies in turn
\be\label{LAX} \za(Z)(\xd f\ot Y)=2(Y|Z)\sD(f)\,, \ee so any Courant algebroid is a Loday algebroid.
{
\subsection{Brackets associated with contact structures}
In \cite{G2}, contact (super)manifolds have been studied as symplectic principal $\R^\times$-bundles $(P,\zw)$; the symplectic form being homogeneous with respect to the $\R^\ti$-action. Similarly, Kirillov brackets on line bundles have been regarded as Poisson principal $\R^\times$-bundles. Consequently, {\em Kirillov algebroids} and {\em contact Courant algebroids} have been introduced, respectively, as homogeneous Lie algebroids and Courant algebroids on vector bundles equipped with a compatible
$\R^\times$-bundle structure. The corresponding brackets are therefore particular Lie algebroid and Courant algebroid brackets, thus Loday algebroid brackets. In other words, Kirillov and contact Courant algebroids are examples of Loday algebroids equipped additionally with some extra geometric structures.

As a canonical example of a contact Courant algebroid, consider the contact 2-manifold represented by the symplectic principal $\R^\times$-bundle $\sT^*[2]\sT[1](\R^\ti\ti M)$, for a purely even manifold $M$ \cite{G2}. As the cubic Hamiltonian $H$ associated with the canonical vector field on $\sT[1](\R^\ti\ti M)$ being the de Rham derivative is 1-homogeneous, we obtain a homogeneous Courant bracket on the linear principal $\R^\ti$-bundle $P=\sT(\R^\ti\ti M)\oplus_{\R^\ti\ti M}\sT^*(\R^\ti\ti M)$. It can be reduced to the vector bundle $E=(\R\ti\sT M)\oplus_M(\R^*\ti\sT^*M)$ whose sections are $(X,f)+(\za,g)$, where $f,g\in\C^\infty(M)$, $X$ is a vector field, and $\za$ is a one-form on $M$,
which is a Loday algebroid bracket of the form
\bea\label{edirac}
&[(X_1,f_1)+(\za_1,g_1),(X_2,f_2)+(\za_2,g_2)]=
\left([X_1,X_2]_{vf},X_1(f_2)-X_2(f_1)\right)\\
&+\left(\Li_{X_1}\za_2-i_{X_2}\xd\za_1+f_1\za_2- f_2\za_1
+f_2\xd g_1+g_2\xd f_1,
X_1(g_2)-X_2(g_1)+i_{X_2}\za_1+f_1g_2\right)\,.\nn
\eea
This is the Dorfman-like version of the bracket whose skew-symmetrization gives exactly the bracket introduced by Wade \cite{Wa0} to define so called {\it $\cE^1(M)$-Dirac structures} and considered also in \cite{GM2}.
The full contact Courant algebroid structure on $E$ consists additionally \cite{G2} of
the symmetric pseudo-Euclidean product
$$\la(X,f)+(\za,g),(X,f)+(\za,g)\ran=\la X,\za\ran+fg\,,$$
and
the vector bundle morphism $\zr^{\! 1}:E\ra\sT M\times\R$, corresponding to a map assigning to sections of $E$ first-order differential operators on $M$, of the form
$$\zr^{\! 1}\left((X,f)+(\za,g)\right)=X+f\,.$$

}

\subsection{Grassmann-Dorfman bracket}
The Dorfman bracket (\ref{CD}) can be immediately generalized to a bracket on sections of $\cT^\we M=\sT
M\oplus_M\wedge\sT^\ast M$, where
$$\wedge\sT^\ast M=\bigoplus_{k=0}^\infty\wedge^k\sT^\ast M\,,$$
so that the module of sections, $\Sec(\wedge\sT^\ast M)=\zW(M)=\bigoplus_{k=0}^\infty\zW^k(M)$, is the
Grassmann algebra of differential forms. The bracket, {\it Grassmann-Dorfman bracket}, is formally given by
the same formula $(\ref{CD})$ and the proof that it is a Loday algebroid bracket is almost the same. The left
anchor is the projection on the summand $\sT M$,
\be\label{CDa1}\zr:\sT M\oplus_M \wedge\sT^*M\ra \sT M\,,
\ee
and
$$\za(Y+\zh)(\xd f\ot(X+\zw))=\xd f\we (i_X\zh+i_Y\zw)=2\,\xd f\we\la X+\zw,Y+\zh\ran_+\,,$$
where
$$\la
X+\zw,Y+\zh\ran_+=\frac{1}{2}\left(i_X\zh+i_Y\zw\right)\,,$$ is a symmetric nondegenerate bilinear form on
$\cT^\we M$, this time with values in $\zW(M)$. Like for the classical Courant-Dorfman bracket, the graph of a
differential form $\zb$ is an isotropic subbundle in $\cT^\we M$ which is involutive (its sections are
closed with respect to the bracket) if and only if $\xd \zb=0$. The Grassmann-Dorfman bracket induces Loday
algebroid brackets on all bundles $\sT M\oplus_M\wedge^k\sT^\ast M$, $k=0,1,\dots,\infty$. These brackets have
been considered in \cite{Sh} and called there {\it higher-order Courant brackets} {(see also \cite{Za})}. Note that this is exactly
the bracket derived from the bracket of first-order (super)differential operators on the Grassmann algebra
$\zW(M)$: we associate with $X+\zw$ the operator $S_{X+\zw}=i_X+\zw\we\,$ and compute the super-commutators,
$$[[S_{X+\zw},\xd]_{sc},S_{Y+\zh}]_{sc}=S_{[X+\zw,Y+\zh]}\,.$$

\subsection{Grassmann-Dorfman bracket for a Lie algebroid}
All the above remains valid when we replace $\sT M$ with a Lie algebroid $(E,[\cdot,\cdot]_E,\zr_E)$, the de
Rham differential $\xd$ with the Lie algebroid cohomology operator $\xd^E$ on $\Sec(\we E^\ast)$, and the Lie
derivative along vector fields with the Lie algebroid Lie derivative $\Ll^E$. We define a bracket on sections
of $E\op_M\we E^\ast$ with formally the same formula
\be\label{CDA}[X+\zw,Y+\zh]=[X,Y]_{E}+\Ll_X^E\zh-i_Y\,\xd^E\zw\,.
\ee
This is a Loday algebroid bracket with the left anchor
$$\zr:E\op_M\we E^\ast\ra\sT M\,,\quad \zr(X+\zw)=\zr_E(X)$$
and
$$\za(Y+\zh)(\xd f\ot(X+\zw))=\xd^E f\we (i_X\zh+i_Y\zw)\,.$$

\subsection{Lie derivative bracket for a Lie algebroid}
The above Loday bracket on sections of $E\op_M\we E^\ast$ has a simpler version. Let us put simply
\be\label{CDA1}[X+\zw,Y+\zh]=[X,Y]_{E}+\Ll_X^E\zh\,.
\ee
This is again a Loday algebroid bracket with the same left anchor and and
$$\za(Y+\zh)(\xd f\ot(X+\zw))=\xd^E f\we i_X\zh+\zr_E(Y)(f)\zw\,.$$
In particular, when reducing to 0-forms, we get a Leibniz algebroid structure on $E\times \R$, where the
bracket is defined by $[X+f,Y+g]=[X,Y]_E+\zr_E(X)g$, the left anchor by $\zr(X,f)=\zr_E(X)$, and the generalized right
anchor by
$$b^r(Y,g)(\xd h\ot (X+f))=-\zr_E(Y)(h)X\,.$$
In other words,
$$\za(Y,g)(\xd h\ot (X+f))=\zr_E(Y)(h)f\,.$$

\subsection{Loday algebroids associated with a Nambu-Poisson structure}

In the following $M$ denotes a smooth $m$-dimensional manifold and $n$ is an integer such that $3\le n\le
m$. An almost Nambu-Poisson structure of order $n$ on $M$ is an $n$-linear bracket $\{\cdot,\ldots,\cdot\}$ on
$\cim$ that is skew-symmetric and has the Leibniz property with respect to the point-wise multiplication. It
corresponds to an $n$-vector field $\zL\in\zG(\we^n\sT M)$. Such a structure is Nambu-Poisson if it verifies
the {\em Filippov identity} ({\em generalized Jacobi identity}): \bea\label{J} &\{ f_1,\dots,f_{n-1},\{
g_1,\dots,g_n\}\}=
\{\{ f_1,\dots,f_{n-1},g_1\},g_2,\dots,g_n\}+\\
&\{ g_1,\{ f_1,\dots,f_{n-1},g_2\},g_3,\dots,g_n\}+\dots+ \{ g_1,\dots,g_{n-1},\{
f_1,\dots,f_{n-1},g_n\}\}\,,\nn \eea i.e., if the Hamiltonian vector fields $X_{f_1\ldots
f_{n-1}}=\{f_1,\ldots,f_{n-1},\cdot\}$ are derivations of the bracket. Alternatively, an almost Nambu-Poisson
structure is Nambu-Poisson if and only if
$$\Ll_{X_{f_1,\ldots,f_{n-1}}}\zL=0\,,$$
for all functions $f_1,\dots,f_{n-1}$.\medskip

Spaces equipped with skew-symmetric brackets satisfying the above identity have been introduced by Filippov
\cite{Fi} under the name {\it $n$-Lie algebras}.\medskip

The concept of Leibniz (Loday) algebroid used in \cite{ILMP} is the usual one, without differentiability
condition for the first argument. Actually, this example is a Loday algebroid in our sense as well. The
bracket is defined for $(n-1)$-forms by
$$[\zw,\zh]=\Ll_{\zr(\zw)}\zh+(-1)^n(i_{\xd\zw}\zL)\zh\,,$$
where
$$\zr:\we^{n-1}\sT^*M\ni\zw\mapsto i_{\zw}\zL\in \sT M$$
is actually the left anchor. Indeed,
$$[\zw,f\zh]=\Ll_{\zr(\zw)}f\zh+(-1)^n(i_{\xd\zw}\zL)f\zh=f[\zw,\zh]+\zr(\zw)(f)\zh\,.$$
For the generalized right anchor we get
$$[f\zw,\zh]=\Ll_{\zr(f\zw)}\zh+(-1)^n(i_{\xd(f\zw)}\zL)\zh=f[\zw,\zh]-i_{\zr(\zw)}(\xd f\we \zh)\,,$$
so
$$\za(\zh)(\xd f\ot\zw)=\zr(\zh)(f)\,\zw-\zr(\zw)(f)\,\zh+\xd f\wedge i_{\zr(\zw)}\zh\,.$$
Note that $\za$ is really a bundle map $\za:\we^{n-1}\sT^\ast M\to \sT M\ot_M\End(\we^{n-1}\sT^\ast M)$, since
it is obviously $C^{\infty}(M)$-linear in $\zh$ and $\zw$, as well as a derivation with respect to
$f.$\medskip

In \cite{Ha,HM}, another Leibniz algebroid associated with the Nambu-Poisson structure $\zL$ is proposed. The
vector bundle is the same, $E=\we^{n-1}\sT^\ast M$, the left anchor map is the same as well,
$\zr(\zw)=i_{\zw}\zL$, but the Loday bracket reads
$$[\zw,\zh]'=\Ll_{\zr(\zw)}\zh-i_{\zr(\zh)}\xd\zw\,.$$
Hence,
\beas[f\zw,\zh]'&=&\Ll_{\zr(f\zw)}\zh-i_{\zr(\zh)}\xd(f\zw)\\
&=&f[\zw,\zh]'-\zr(\zh)(f)\,\zw+\xd f\we(i_{\zr(\zw)}\zh+i_{\zr(\zh)}\zw)\,, \eeas so that for the generalized
right anchor we get
$$\za(\zh)(\xd f\ot\zw)=\xd f\we(i_{\zr(\zw)}\zh+i_{\zr(\zh)}\zw)\,.$$
This Loday algebroid structure is clearly the one obtained from the Grassmann-Dorfman bracket on the graph of
$\zL$,
$$\operatorname{graph}(\zL)=\{ \zr(\zw)+\zw:\zw\in\zW^{n-1}(M)\}\,.$$
Actually, an $n$-vector field $\zL$ is a Nambu-Poisson tensor if and only if its graph is closed with respect
to the Grassmann-Dorfman bracket {\cite{Sh,Ha}}.

\section{The Lie pseudoalgebra of a Loday algebroid}
Let us fix a Loday pseudoalgebra bracket $[\cdot,\cdot]$ on an $\cA$-module $\cE$. Let $\zr:\cE\ra \Der(\cA)$
be the left anchor map, and let
$$b^r=\zr-\za:\cE\ra\Der(\cA)\ot_\cA\End(\cE)$$
be the generalized right anchor map. For every $X\in\cE$ we will view $\za(X)$ as a $\cA$-module homomorphism
$\za(X):\zW^1\ot_\cA\cE\ra\cE$, where $\zW^1$ is the $\cA$-submodule of $\Hom_\cA(\cE;\cA)$ generated by
$\xd\cA=\{\xd f:f\in\cA\}$ and $\xd f(D)=D(f)$.

It is a well-known fact that the subspace $\mathfrak{g}^{0}$ generated in a Loday algebra $\mathfrak{g}$ by
the symmetrized brackets $X\diamond Y=[X,Y]+[Y,X]$ is a two-sided ideal and that $\mathfrak{g}/\mathfrak{g}^0$
is a Lie algebra. Putting
$$\cE^0=\operatorname{span}\{ [X,X]: X\in\cE\}\,,$$
we have then \be\label{ce} [\cE^0,\cE]=0\,,\quad [\cE,\cE^0]\subset\cE^0\,. \ee Indeed, symmetrized brackets
are spanned by squares $[X,X]$, so, due to the Jacobi identity,
$$[[X,X],Y]=[X,[X,Y]]-[X,[X,Y]]=0$$
and \be\label{dia}[Y,[X,X]]=[[Y,X],X]+[X,[Y,X]]=[X,Y]\diamond Y\,.
\ee However, working with $\cA$-modules, we would like to have an
$\cA$-module structure on $\cE/\cE^0$. Unfortunately, $\cE^0$ is not a submodule in general. Let us consider
therefore the $\cA$-submodule $\bcE$ of $\cE$ generated by $\cE^0$, i.e., $\bcE=\cA\cdot\cE^0$.
\begin{lem} For all $f\in\cA$ and $X,Y,Z\in\cE$ we have
\bea\label{fid0}\za(X)(\xd f\ot Y)&=&X\diamond(fY)-f(X\diamond Y)\,,\\
{[\za(X)(\xd f\ot Y),Z]}&=&{\zr(Z)(f)(X\diamond Y)-\za(Z)(\xd f\ot(X\diamond Y))\,.}\label{fid1} \eea In
particular,
\be\label{sy} [\za(X)(\xd f\ot Y),Z]=[\za(Y)(\xd f\ot X),Z]\,. \ee
\end{lem}
\begin{proof} To prove (\ref{fid0}) it suffices to combine the identity $[X,fY]=f[X,Y]+\zr(X)(f)(Y)$ with
$$[fY,X]=f[Y,X]-\zr(X)(f)Y+\za(X)(\xd f\ot Y)\,.$$
Then, as $[\cE^0,\cE]=0$,
$$[\za(X)(\xd f\ot Y),Z]=-[f(X\diamond Y),Z]={\zr(Z)(f)(X\diamond Y)-\za(Z)(\xd f\ot(X\diamond Y))\,.}$$
\end{proof}
\begin{cor}\label{cor1} For all $f\in\cA$ and $X,Y\in\cE$,
\be\label{fid} \za(X)(\xd f\ot Y)\in\bcE\,, \ee and the left
anchor vanishes on $\bcE$, \be\label{qan} \zr(\bcE)=0\,. \ee Moreover, $\bcE$ is a two-sided Loday ideal in
$\cE$ and the Loday bracket induces on the $\cA$-module $\bar\cE=\cE/\bcE$  a Lie pseudoalgebra structure with
the anchor
\be\label{qan1}\bar\zr([X])=\zr(X)\,, \ee where $[X]$ denotes the
coset of $X$.
\end{cor}
\begin{proof}
The first statement follows directly from (\ref{fid0}). As $[\cE^0,\cE]=0$, the anchor vanishes on $\cE^0$ and
thus on $\bcE=\cA\cdot\cE^0$. From
$$[Z,f(X\diamond Y)]=f[Z,X\diamond Y]+\zr(X)(f)(X\diamond Y)\in\bcE$$
and
$$[f(X\diamond Y),Z]=f[(X\diamond Y),Z]-\zr(Z)(f)(X\diamond Y)+\za(Z)(\xd f\ot(X\diamond Y))\in\bcE\,,$$
we conclude that $\bcE$ is a two-sided ideal. As $\bcE$ contains all elements $X\diamond Y$, The Loday bracket
induces on $\cE/\bcE$ a skew-symmetric bracket with the anchor (\ref{qan1}) and satisfying the Jacobi
identity, thus a Lie pseudoalgebra structure.
\end{proof}
\begin{definition} The Lie pseudoalgebra $\bar\cE=\cE/\bcE$ we will call the {\it Lie pseudoalgebra of the Loday pseudoalgebra $\cE$}.
If $\cE=\Sec(E)$ is the Loday pseudoalgebra of a Loday algebroid on a vector bundle $E$ and the module $\bcE$ is the module
of sections of a vector subbundle $\bar E$ of $E$, we deal with the {\it Lie algebroid of the Loday algebroid $E$}.
\end{definition}
\begin{example} The Lie algebroid of the Courant-Dorfman bracket is the canonical Lie algebroid $\sT M$.
\end{example}
\begin{thm} For any Loday pseudoalgebra structure on an $\cA$-module $\cE$ there is a short exact sequence of morphisms of
Loday pseudoalgebras over $\cA$,
\be\label{es}
0\longrightarrow\bcE\longrightarrow\cE\longrightarrow\bar\cE\longrightarrow 0\,, \ee where $\bcE$ -- the
$\cA$-submodule in $\cE$ generated by $\{[X,X]:X\in\cE\}$ -- is a Loday pseudoalgebra with the trivial left
anchor and $\bar\cE=\cE/\bcE$ is a Lie pseudoalgebra.
\end{thm}
Note that the Loday ideal $\cE^0$ is clearly commutative, while the modular ideal $\bcE$ is no longer
commutative in general.

\section{Loday algebroid cohomology}

We first recall the definition of the Loday cochain complex associated to a bi-module over a Loday algebra
\cite{LP}.\medskip

Let $\mathbb{K}$ be a field of nonzero characteristic and $V$ a $\mathbb{K}$-vector space endowed with a
(left) Loday bracket $[\cdot,\cdot]$. A {\it bimodule} over a Loday algebra $(V,[\cdot,\cdot])$ is a $\mathbb{K}$-vector space
$W$ together with a left (resp., right) {\it action} $\zm^l\in\sss{Hom}(V\otimes W,W)$ (resp.,
$\zm^r\in\sss{Hom}(W\otimes V,W)$) that verify the following requirements \be\label{VVW}
\zm^r[x,y]=\zm^r(y)\zm^r(x)+\zm^l(x)\zm^r(y),\ee \be\label{WVV}
\zm^r[x,y]=\zm^l(x)\zm^r(y)-\zm^r(y)\zm^l(x),\ee \be\label{VWV}
\zm^l[x,y]=\zm^l(x)\zm^l(y)-\zm^l(y)\zm^l(x),\ee for all $x,y\in V.$\medskip

The {\it Loday cochain complex} associated to the Loday algebra $(V,[\cdot,\cdot])$ and the bimodule $(W,\zm^l,\zm^r)$, shortly -- to $B=([\cdot,\cdot],\zm^r,\zm^l)$,
is made up by the cochain space $$\Lin^\bullet(V,W)=\bigoplus_{p\in \mathbb{N}}\, \Lin^p(V,W)=\bigoplus_{p\in
\mathbb{N}}\,\sss{Hom}(V^{\otimes p},W),$$ where we set $\Lin^0(V,W)=W$, and the coboundary operator $\partial_B$
defined, for any $p$-cochain $c$ and any vectors $x_1,\ldots,x_{p+1}\in V$, by
\bea\nn (\pa_B
c)(x_1,\ldots,x_{p+1})&=&(-1)^{p+1}\zm^r(x_{p+1})c(x_1,\ldots,x_p)+\sum_{i=1}^p(-1)^{i+1}\zm^l(x_i)c(x_1,\ldots\hat{\imath}\ldots,x_{p+1})\\
&&+\sum_{i<j}(-1)^{i}c(x_1,\ldots\hat{\imath}\ldots,\stackrel{(j)}{\overbrace{[x_i,x_j]}},\ldots,x_{p+1})\;\;.\label{LodCohOp}\eea

Let now $\zr$ be a {\it representation} of the Loday algebra $(V,[\cdot,\cdot])$ on a $\mathbb{K}$-vector space $W$,
i.e. a Loday algebra homomorphism $\zr: V\to \sss{End}(W)$. It is easily checked that $\zm^l:=\zr$ and
$\zm^r:=-\zr$ endow $W$ with a bimodule structure over $V$. Moreover, in this case of a bimodule induced by a
representation, the Loday cohomology operator reads

\bea\label{LodCohOpRepr}(\pa_B
c)(x_1,\ldots,x_{p+1})&=&\sum_{i=1}^{p+1}(-1)^{i+1}\zr(x_i)c(x_1,\ldots\hat{\imath}\ldots,x_{p+1})\\
&&
+\sum_{i<j}(-1)^ic(x_1,\ldots\hat{\imath}\ldots,\stackrel{(j)}{\overbrace{[x_i,x_j]}},\ldots,x_{p+1})\;\;.\nn
\eea

Note that the above operator $\pa_B$ is well defined if only the {\it
map} $\zr:V\to \sss{End}(W)$ and the {\it bracket}
$[\cdot,\cdot]:V\otimes V\to V$ are given. We will refer to it as
to the {\it Loday operator} associated with $B=([\cdot,\cdot],\zr)$.
The point is that $\pa_B^2=0$ if and only if $[\cdot,\cdot]$ is a
Loday bracket and $\zr$ is its representation. Indeed, the Loday
algebra homomorphism property of $\zr$ (resp., the Jacobi identity
for $[\cdot,\cdot]$) is encoded in $\pa_B^2=0$ on
$\sss{Lin}^0(V,W)=W$ (resp., $\sss{Lin}^1(V,W)$), at least if
$W\neq \{0\}$, what we assume).\medskip

{
Let now $E$ be a vector bundle over a manifold $M$ and $B=([\cdot,\cdot],\zr)$ be an {\it anchored
faint algebroid} structure on $E$, where $[\cdot,\cdot]$ is a
faint pseudoalgebra bracket (bidifferential operator) and $\zr:E\ra \sT M$ is
a vector bundle morphism covering the identity, so inducing a
module morphism $\zr:\sss{Sec}(E)\ra\Der(C^\infty(M))=\cX(M)$.
It is easy to see that, unlike in the case of a Lie algebroid,
the tensor algebra of sections of $\oplus_{k=0}^\infty(E^*)^{\otimes k}$ is, in general, not
invariant under the Loday cohomology operator $\pa_B$ associated with
$B=([\cdot,\cdot],\zr)$. Actually, $\pa_B$ rises the degree of a multidifferential operator by one, even when the Loday bracket is skew-symmetric (see e.g. \cite{Ldd,Ldd1}).

\begin{example} \cite{Ldd1}
Suppose that $M$ is a Riemannian manifold with metric tensor
$g$ and let $\pa_B$ be the Loday coboundary operator associated with the canonical
bracket of vector fields $B=([\cdot,\cdot]_{\text{vf}},\text{id}_{\sT M})$ on $E=\sT M$. {When adopting the conventions of \cite{Ldd1}, where the Loday differential associated to right Loday algebras is considered, we then get,} for all $X,Y,Z\in\X(M)$,
\be\label{LC}(\pa_B g)(X,Y,Z)=2g(Y,\nabla_XZ)\,,
\ee
where $\nabla$ is the Levi-Civita connection on M. One can say that the Loday differential of a Riemannian metric defines the corresponding Levi-Civita connection, which clearly is no longer a tensor on $M$.
\end{example}

The above observation suggests to consider in $\Lin^\bullet(\sss{Sec}(E),C^{\infty}(M))$, instead of {$\sss{Sec}(\otimes^{\bullet}E^*)$},
the subspace
$$\cD^\bullet(\sss{Sec}(E),C^{\infty}(M))\subset\Lin^\bullet(\sss{Sec}(E),C^{\infty}(M))$$
consisting of all multidifferential operators. If now $B=([\cdot,\cdot],\zr)$ is an {\it anchored
{faint} algebroid} structure on $E${, see above}, then it is clear that the space
{$\cD^\bullet(E):=\cD^\bullet(\sss{Sec}(E),C^{\infty}(M))$} is stable under the
Loday operator $\pa_B$ associated with
$B=([\cdot,\cdot],\zr)$.}

\medskip
In particular, if $([\cdot,\cdot],\zr,\za)$ is a Loday algebroid structure on
$E$, its left anchor $\zr:\sss{Sec}(E)\to
\sss{Der}(C^{\infty}(M))\subset \sss{End}(C^{\infty}(M))$ is a
representation of the Loday algebra $(\sss{Sec}(E),[\cdot,\cdot])$
by derivations on $C^{\infty}(M)$ and $\pa_B^2=0$, so $\pa_B$ is a
coboundary operator.

\begin{definition} Let $(E,[\cdot,\cdot],\zr,\za)$ be a Loday algebroid over a manifold $M$. We call {\it
Loday algebroid cohomology}, the cohomology of the Loday cochain
subcomplex {$(\cD^\bullet(E),\pa_B)$}
associated with $B=([\cdot,\cdot],\zr)$, i.e. the Loday algebra
structure $[\cdot,\cdot]$ on $\sss{Sec}(E)$ represented by $\zr$
on $C^{\infty}(M)$.\end{definition}

\section{Supercommutative geometric interpretation}

Let $E$ be a vector bundle over a manifold $M$.
{Looking for a canonical superalgebra structure in $\cD^\bullet(E)$, a natural candidate is the {\it shuffle (super)product}, introduced by Eilenberg and Mac Lane \cite{EML}
(see also \cite{R1,R2}). It is known that a shuffle algebra on a free associative algebra is a free commutative algebra with the Lyndon words as its free generators \cite{Ra}. A similar result is valid in the supercommutative case \cite{ZM}. In this sense the free shuffle superalgebra represents a supercommutative space.
}
\begin{definition} For any $\ell'\in
\cD^p(\sss{Sec}(E),C^{\infty}(M))$ and $\ell''\in \cD^q(\sss{Sec}(E),C^{\infty}(M))$,
$p,q\in\mathbb{N}$, we define the {\it shuffle product}
$$(\ell'\pitchfork
\ell'')(X_1,\ldots,X_{p+q}):=\sum_{\zs\in\ope{sh}(p,q)}\ope{sign}\zs\,\;
\ell'(X_{\zs_1},\ldots,X_{\zs_p})\;\ell''(X_{\zs_{p+1}},\ldots,X_{\zs_{p+q}}),$$ where the $X_i$-s denote
sections in $\sss{Sec}(E)$ and where $\ope{sh}(p,q)\subset \mathbb{S}_{p+q}$ is the subset of the symmetric
group $\mathbb{S}_{p+q}$ made up by all $(p,q)$-shuffles.\end{definition}

The next proposition is well-known.

\begin{prop} The space {$\cD^\bullet(E)$}, together with the shuffle multiplication $\pitchfork$, is a graded commutative associative unital $\R$-algebra.\end{prop}

We refer to this algebra as the {\it shuffle algebra of the vector bundle} $E\to M$, or simply, of $E$.

\medskip
Let $B=([\cdot,\cdot],\zr)$ be an anchored {faint} algebroid
structure on $E$ and let $\pa_B$ be the associated Loday operator
in $\cD^\bullet(E)$. Note that we would have $\pa_B^2=0$ if we had assumed that we deal with a Loday
algebroid.\medskip

Denote now by $\sD^k(E)$ those $k$-linear multidifferential
operators from $\cD^k(E)$ which are of degree 0 with respect to
the last variable and of total degree $\le k-1$, and set
{$\sD^\bullet(E)=\bigoplus_{k=0}^\infty\sD^k(E)$}. By convention,
$\sD^0(E)=\cD^0(E)=C^\infty(M)$. Moreover,
$\sD^1(E)=\Sec(E^\ast)$. It is easy to see that $\sD^\bullet(E)$
is stable for the shuffle multiplication. We will call the
subalgebra $(\sD^\bullet(E),\pitchfork)$, the {\it reduced shuffle
algebra}, and refer to the corresponding graded ringed space as
{\it supercommutative manifold}. Let us emphasize that this
denomination is in the present text merely a terminological
convention. The graded ringed spaces of the considered type are
being investigated in a separate work.

\begin{thm}\label{GeoIntKLA1} The coboundary operator $\pa_B$ is a degree 1
graded derivation of the shuffle algebra of $E$, i.e.
\be\pa(\ell'\pitchfork\ell'')=(\pa \ell')\pitchfork \ell''+(-1)^p
\ell'\pitchfork (\pa\ell''),\label{DerShuffle}\ee for any
$\ell'\in \cD^p(E)$ and $\ell''\in \cD^q(E)$. Moreover, if
$[\cdot,\cdot]$ is a pseudoalgebra bracket, i.e., if it is of
total order $\le 1$ and $\zr$ is the left anchor for
$[\cdot,\cdot]$, then $\pa_B$ leaves invariant the reduced shuffle
algebra $\sD^\bullet(E)\subset\cD^\bullet(E)$.\end{thm}

The claim is easily checked on low degree examples. The general proof is as follows.

\begin{proof} The value of the {\small LHS} of Equation (\ref{DerShuffle}) on sections $X_1,\ldots,X_{p+q+1}\in\sss{Sec}(E)$
is given by $S_1 +\ldots + S_4$, where
$$S_1=\sum_{k=1}^{p+1}\sum_{\zt\in\ope{sh}(p,q)}(-1)^{k+1}\ope{sign}\zt\; \zr(X_k)\left(\ell'(X_{\zt_1},\ldots,
\widehat{X}_{\zt_k},\ldots, X_{\zt_{p+1}})\;\ell''(X_{\zt_{p+2}},\ldots, X_{\zt_{p+q+1}})\right)$$
and
$$S_3=\sum_{1\le k< m\le p+q+1}\sum_{\zt\in \ope{sh}(p,q)}(-1)^k\ope{sign}\zt\;\,\ell'(X_{\zt_1},\ldots,[X_k,X_m],\ldots)
\;\ell''(X_{\zt_{-}},\ldots).$$
In the sum $S_2$, which is similar to $S_1$, the index $k$ runs through $\{p+2,\ldots,p+q+1\}$ ($X_{\zt_k}$ is
then missing in $\ell''$). The sum $S_3$ contains those shuffle permutations of $1\ldots \hat{k}\ldots p+q+1$
that send the argument $[X_k,X_m]$ with index $m=:\zt_r$ into $\ell'$, whereas $S_4$ is taken over the shuffle
permutations that send $[X_k,X_m]$ into $\ell''.$\medskip

Analogously, the value of $(\pa\ell')\pitchfork\ell''$ equals $T_1+T_2$ with
$$T_1=\sum_{\zs\in\ope{sh}(p+1,q)}\sum_{i=1}^{p+1}\ope{sign}\zs\,(-1)^{i+1}\left(\zr(X_{\zs_i})\,\ell'(X_{\zs_1},
\ldots,\widehat{X}_{\zs_i},\ldots,X_{\zs_{p+1}})\right)\ell''(X_{\zs_{p+2}},\ldots,X_{\zs_{p+q+1}})$$
and
$$T_2=\sum_{\zs\in\ope{sh}(p+1,q)}\sum_{1\le i<j\le p+1}\ope{sign}\zs\,(-1)^{i}\,\ell'(X_{\zs_1},\ldots,[X_{\zs_i},
X_{\zs_j}],\ldots)\;\ell''(X_{\zs_{p+2}},\ldots,X_{\zs_{p+q+1}})$$
(whereas the value $T_3+T_4$ of $(-1)^p\;\ell'\pitchfork(\pa\ell'')$, which is similar, is not (really) needed
in this (sketch of) proof).\medskip

Let us stress that in $S_3$ and $T_2$ the bracket is in its natural position determined by the index $\zt_r=m$
or $\zs_j$ of its second argument, that, since $\ope{sh}(p,q)\simeq \mathbb{S}_{p+q}/(\mathbb{S}_p\times
\mathbb{S}_q)$, the number of $(p,q)$-shuffles equals ${(p+q)!}/{(p!\,q!)}\,$, and that in $S_1$ the vector field
$\zr(X_k)$ acts on a product of functions according to the Leibniz rule, so that each term splits. It is now
easily checked that after this splitting the number of different terms in $\zr(X_{-})$ (resp. $[X_-,X_-]$) in
the {\small LHS} and the {\small RHS} of Equation (\ref{DerShuffle}) is equal to $2 (p+q+1)!/(p!\,q!)$ (resp.
$(p+q)(p+q+1)!/(2\,p!\,q!)$). To prove that both sides coincide, it therefore suffices to show that any term of the
{\small LHS} can be found in the {\small RHS}.\medskip

We first check this for any split term of $S_1$ with vector field action on the value of $\ell'$ (the proof is
similar if the field acts on the second function and also if we choose a split term in $S_2$),
$$(-1)^{k+1}\ope{sign}\zt\; \left(\zr(X_k)\ell'(X_{\zt_1},\ldots, \widehat{X}_{\zt_k},\ldots,
X_{\zt_{p+1}})\right)\,\ell''(X_{\zt_{p+2}},\ldots, X_{\zt_{p+q+1}}),$$ where $k\in\{1,\ldots,p+1\}$ is fixed,
as well as $\zt\in\ope{sh(p,q)}$ -- which permutes $1\ldots\hat{k}\ldots p+q+1$. This term exists also in
$T_1$. Indeed, the shuffle $\zt$ induces a unique shuffle $\zs\in\ope{sh}(p+1,q)$ and a unique
$i\in\{1,\ldots,p+1\}$ such that $\zs_i=k.$ The corresponding term of $T_1$ then coincides with the chosen
term in $S_1$, since, as easily seen, $\ope{sign}\zs\;(-1)^{i+1}=(-1)^{k+1}\ope{sign}\zt$.\medskip

Consider now a term in $S_3$ (the proof is analogous for the terms of $S_4$),
$$(-1)^k\ope{sign}\zt\;\,\ell'(X_{\zt_1},\ldots,[X_k,X_m],\ldots)\;\ell''(X_{\zt_{-}},\ldots),$$
where $k<m$ are fixed in $\{1,\ldots, p+q+1\}$ and where $\zt\in \ope{sh}(p,q)$ is a fixed permutation of
$1\ldots\hat{k}\ldots p+q+1$ such that the section $[X_k,X_m]$ with index $m=:\zt_r$ is an argument of
$\ell'$. The shuffle $\zt$ induces a unique shuffle $\zs\in\ope{sh}(p+1,q)$. Set $k=:\zs_i$ and $m=:\zs_j$. Of
course $1\le i<j\le p+1$. This means that the chosen term reads
$$(-1)^k\ope{sign}\zt\;\,\ell'(X_{\zs_1},\ldots,[X_{\zs_i},X_{\zs_j}],\ldots,X_{\zs_{p+1}})\;\ell''(X_{\zs_{p+2}},\ldots,X_{\zs_{p+q+1}}).$$
Finally this term is a term of $T_2$, as it is again clear that
$(-1)^k\ope{sign}\zt=\ope{sign}\zs\,(-1)^i$.\medskip

That $\sD^\bullet(E)$ is invariant under $\pa_B$ in the case of a
pseudoalgebra bracket is obvious. This completes the
proof.\end{proof}

Note that the derivations $\pa_B$ of the reduced shuffle algebra
(in the case of pseudoalgebra brackets on $\Sec(E)$) are, due to
formula (\ref{LodCohOpRepr}), completely determined by their
values on $\sD^0(E)\oplus\sD^1(E)$. More precisely,
$B=([\cdot,\cdot],\zr)$ can be easily reconstructed from $\pa_B$
thanks to the formulae \be\label{anchor-reconstruction}
\zr(X)(f)=\la X,\pa_Bf\ran \ee and
\be\label{bracket-reconstruction} \la\mathfrak{l},[X,Y]\ran=\la
X,\pa_B\la\mathfrak{l},Y\ran\ran-\la
Y,\pa_B\la\mathfrak{l},X\ran\ran-\pa_B\mathfrak{l}(X,Y)\,, \ee
where $X,Y\in\Sec(E)$, $\mathfrak{l}\in\Sec(E^\ast)$, and $f\in
C^\infty(M)$.

\begin{thm}\label{th:LO} If $\pa$ is a derivation of the reduced shuffle algebra $\sD^\bullet(E)$, then on $\sD^0(E)\oplus\sD^1(E)$ the
derivation $\pa$ coincides with $\pa_B$ for a certain uniquely determined
$B=([\cdot,\cdot]_\pa,\zr_\pa)$ associated with a pseudoalgebra
bracket $[\cdot,\cdot]_\pa$ on $\Sec(E)$.
\end{thm}
\begin{proof}
Let us define $\zr=\zr_\pa$ and $[\cdot,\cdot]=[\cdot,\cdot]_\pa$ out of formulae
(\ref{anchor-reconstruction}) and (\ref{bracket-reconstruction}),
i.e., \be\label{anchor-reconstruction1} \zr(X)(f)=\la X,\pa
f\ran \ee and \be\label{bracket-reconstruction1}
\la\mathfrak{l},[X,Y]\ran=\la
X,\pa\la\mathfrak{l},Y\ran\ran-\la
Y,\pa\la\mathfrak{l},X\ran\ran-\pa\mathfrak{l}(X,Y)\,. \ee The
fact that $\zr(X)$ is a derivation of $C^{\infty}(M)$ is a
direct consequence of the shuffle algebra derivation property of
$\pa$. Eventually, the map $\zr$ is visibly associated with
a bundle map $\zr:E\to \sss{T}M$.\smallskip

The bracket $[\cdot,\cdot]$ has $\zr$ as left anchor.
Indeed, since $\pa\mathfrak{l}(X,Y)$ is of order 0 with respect to
$Y$, we get from (\ref{bracket-reconstruction1})
$$[X,fY]-f[X,Y]= \la X,\pa f\ran Y=\zr(X)(f)Y\,.$$
Similarly, as $\pa\mathfrak{l}(X,Y)$ is of order 1 with respect to $X$ and of order 0 with respect to $Y$, the operator
$$\zd_1(f)\left(\pa\mathfrak{l}\right)(X,Y)=\pa\mathfrak{l}(fX,Y)-f\pa\mathfrak{l}(X,Y)$$
is $C^\infty(M)$-bilinear, so that the LHS of
$$\la\mathfrak{l},[fX,Y]-f[X,Y]\ran=-\la Y,\pa f\ran\la\mathfrak{l},X\ran -\zd_1(f)\left(\pa\mathfrak{l}\right)(X,Y),$$ see (\ref{bracket-reconstruction1}),
is $C^\infty(M)$-linear with respect to $X$ and $Y$ and a
derivation with respect to $f$. The bracket $[\cdot,\cdot]$ is
therefore of total order $\le 1$ with the generalized right anchor
$b\,^r=\zr-\za$, where $\za$ is determined by
the identity \be\label{za}\la\mathfrak{l},\za(Y)(\xd f\ot
X)\ran=\zd_1(f)\left(\pa\mathfrak{l}\right)(X,Y)\,. \ee This
corroborates that $\za$ is a bundle map from $E$ to
$\sss{T}M\otimes_M\sss{End}(E)$.\end{proof}

\begin{definition} Let ${\ope{Der}}_1(\sD^\bullet(E),\pitchfork)$ be the space of degree $1$ graded
derivations $\pa$ of the reduced shuffle algebra that verify, for any $c\in \sD^2(E)$ and any
$X_i\in\sss{Sec}(E)$, $i=1,2,3$,
\bea\label{EncodingJacobi}(\pa c)(X_1,X_2,X_3)&=&\sum_{i=1}^3(-1)^{i+1}\la\pa
(c(X_1,\ldots\hat{\imath}\ldots,X_3)),X_i\rangle\\
&&+\sum_{i<j}(-1)^i\,c(X_1,\ldots\hat{\imath}\ldots,\stackrel{(j)}
{[X_i,X_j]_\pa},\ldots,X_3)\,.\nn
\eea
A {\it homological vector field} of the supercommutative manifold $(M,\sD^\bullet(E))$ is a square-zero
derivation in $\ope{{Der}}_1(\sD^\bullet(E),\pitchfork)$. Two homological vector fields of $(M,\sD^\bullet(E))$ are {\it equivalent}, if they coincide on $C^{\infty}(M)$ and on
$\sss{Sec}(E^*)$.
\end{definition}

\medskip
Observe that Equation (\ref{EncodingJacobi}) implies that two equivalent homological fields also coincide on
$\sD^2(E)$.
We are now prepared to give the main theorem of this section.

\begin{thm}\label{GeoIntKLA2} Let $E$ be a vector bundle. There exists a 1-to-1 correspondence between equivalence classes of homological vector
fields
$$\pa\in\ope{{Der}}_1(\sD^\bullet(E),\pitchfork),\;\pa^2=0$$
and Loday algebroid structures on $E$.\end{thm}

\begin{remark} This theorem is a kind of a non-antisymmetric counterpart of
the well-known similar correspondence between homological vector fields of split supermanifolds and Lie
algebroids. Furthermore, it may be viewed as an analogue for Loday algebroids of the celebrated
Ginzburg-Kapranov correspondence for quadratic Koszul operads \cite{GK}. According to the latter result,
homotopy Loday structures on a graded vector space $V$ correspond bijectively to degree 1 differentials of the
Zinbiel algebra $(\bar{\otimes}sV^*,\star)$, where $s$ is the suspension operator and where
$\bar{\otimes}sV^*$ denotes the reduced tensor module over $sV^*.$ However, in our geometric setting scalars,
or better functions, must be incorporated (see the proof of Theorem \ref{GeoIntKLA2}), which turns out to be
impossible without passing from the Zinbiel multiplication or half shuffle $\star$ to its symmetrization
$\pitchfork$. Moreover, it is clear that the algebraic structure on the function sheaf should be
associative.\end{remark}

\begin{proof} Let $([\cdot,\cdot],\zr,\za)$ be a Loday algebroid
structure on the given vector bundle $E\to M.$ According to
Theorem \ref{GeoIntKLA1}, the corresponding coboundary operator
$\partial_B$ is a square 0 degree 1 graded derivation of the
reduced shuffle algebra and (\ref{EncodingJacobi}) is satisfied by
definition, as $[\cdot,\cdot]_{\pa_B}=[\cdot,\cdot]$.\medskip

Conversely, let $\pa$ be such a homological vector field.
According to Theorem \ref{th:LO}, the derivation $\pa$ coincides
on $\sD^0(E)\oplus\sD^1(E)$ with $\pa_B$ for a certain
pseudoalgebra bracket $[\cdot,\cdot]=[\cdot,\cdot]_\pa$ on
$\Sec(E)$. Its left anchor is $\zr=\zr_\pa$ and the generalized
right anchor $b^r=\zr-\za$ is determined by means of formula
(\ref{za}), where $\mathfrak{l}$ runs through all sections of
$E^\ast$.\medskip

To prove that the triplet $([\cdot,\cdot],\zr,\za)$ defines a Loday algebroid structure on $E,$ it now suffices to
check that the Jacobi identity holds true. It follows from (\ref{bracket-reconstruction1}) that
$$\la\mathfrak{l},[X_1,[X_2,X_3]]\ran=-\la\pa\la \mathfrak{l},X_1\rangle, [X_2,X_3]\rangle+\la\pa\la
\mathfrak{l},[X_2,X_3]\rangle, X_1\rangle-(\pa\mathfrak{l})(X_1,[X_2,X_3]).$$ Since the first term of the
{\small RHS} is (up to sign) the evaluation of $[X_2,X_3]$ on the section $\pa\la \mathfrak{l},X_1\rangle$ of
$E^*$, and a similar remark is valid for the contraction $\la \mathfrak{l},[X_2,X_3]\rangle$ in the second
term, we can apply (\ref{bracket-reconstruction1}) also to these two brackets. If we proceed analogously for
$[[X_1,X_2],X_3]$ and $[X_2,[X_1,X_3]]$, and use (\ref{anchor-reconstruction1}) and the homological property
$\pa^2=0$, we find, after simplification, that the sum of the preceding three double brackets equals
$$\sum_{i=1}^3(-1)^{i+1}
\zr(X_i)(\pa\mathfrak{l})(X_1,\ldots\hat{\imath}\ldots,X_3)+\sum_{i<j}(-1)^i\,(\pa\mathfrak{l})(X_1,\ldots\hat{\imath}\ldots,\stackrel{(j)}
{\overbrace{[X_i,X_j]}},\ldots,X_3)\;.
$$
In view of (\ref{EncodingJacobi}), the latter expression coincides with $(\pa^2\mathfrak{l})(X_1,X_2,X_3)=0$, so that
the Jacobi identity holds.\medskip

It is clear that the just detailed assignment of a Loday algebroid
structure to any homological vector field can be viewed as a map
on equivalence classes of homological vector fields.\end{proof}

Having a homological vector field $\pa$ associated with a Loday
algebroid structure $([\cdot,\cdot],\zr,\za)$ on $E$, we can
easily develop the corresponding Cartan calculus for the shuffle
algebra $\cD^\bullet(E)$.

\begin{prop} For any $X\in\Sec(E)$, the contraction
$$\cD^p(E)\ni\ell\mapsto i_X\ell\in\cD^{p-1}(E)\,,\quad (i_X\ell)(X_1,\dots,X_{p-1})=\ell(X,X_1,\dots,X_{p-1})\,,$$
is a degree $-1$ graded derivation of the shuffle algebra $({\cal
D}^{\bullet}(E),\pitchfork)$.\end{prop}

\begin{proof} Using usual notations, our definitions, as well as a separation of
the involved shuffles $\zs$ into the $\zs$-s that verify
${\zs_1}=1$ and those for which ${\zs_{p+1}=1}$, we get
$$\left(i_{X_1}(\ell'\pitchfork\ell'')\right)(X_2,\ldots,X_{p+q})= \sum_{\zs:\zs_1=1}\ope{sign}\zs\;(i_{X_1}\ell')(X_{\zs_2},
\ldots,X_{\zs_{p}}) \ell''(X_{\zs_{p+1}},\ldots,X_{\zs_{p+q}})$$
$$+\sum_{\zs:\zs_{p+1}=1}\ope{sign}\zs\;\ell'(X_{\zs_1},\ldots,X_{\zs_{p}})
(i_{X_1}\ell'')(X_{\zs_{p+2}},\ldots,X_{\zs_{p+q}}).$$ Whereas a
$(p,q)$-shuffle of the type $\zs_1=1$ is a $(p-1,q)$-shuffle with
same signature, a $(p,q)$-shuffle such that $\zs_{p+1}=1$ defines
a $(p,q-1)$-shuffle with signature $(-1)^p\ope{sign}\zs$.
Therefore, we finally get
$$i_{X_1}(\ell'\pitchfork\ell'')=(i_{X_1}\ell')\pitchfork
\ell''+(-1)^p\ell'\pitchfork(i_{X_1}\ell'').$$\end{proof}

Observe that the supercommutators
$[i_X,i_Y]_{\ope{sc}}=i_Xi_Y+i_Yi_X$ do not necessarily vanish, so
that the derivations $i_X$ of the shuffle algebra generate a Lie
superalgebra of derivations with negative degrees. Indeed,
$[i_X,i_Y]_{\ope{sc}}=:i_{X\Box Y}$,
$[[i_X,i_Y]_{\ope{sc}},i_Z]_{\ope{sc}}=:i_{(X\Box Y)\Box Z},...$
are derivations of degree $-2$, $-3,...$ given on any
$\ell\in{\cal D}^p(E)$ by
$$(i_{X\Box Y}\ell)(X_1,\dots,X_{p-2})=\ell(Y,X,X_1,\dots,X_{p-2})+\ell(X,Y,X_1,\dots,X_{p-2})\,,$$
$$(i_{(X\Box X)\Box
Y}\ell)(X_1,\dots,X_{p-3})=2\ell(Y,X,X,X_1,\dots,X_{p-3})-2\ell(X,X,Y,X_1,\dots,X_{p-3})\,,...$$\smallskip

The next proposition is obvious.

\begin{prop} The supercommutator $\cL_X:=[\pa,i_X]_{\ope{sc}}=\pa
i_X+i_X\pa$, $X\in\Sec(E)$, is a degree 0 graded derivation of the
shuffle algebra. Explicitly, for any $\ell\in\cD^p(E)$ and
$X_1,\ldots,X_p\in\Sec(E)$, \be\label{Ld}
(\cL_X\ell)(X_1,\dots,X_p)=\zr(X)\left(\ell(X_1,\dots,X_p)\right)-\sum_i\ell(X_1,\dots,\stackrel{(i)}
{\overbrace{[X,X_i]}},\dots,X_p)\,.\ee\end{prop}

We refer to the derivation $\cL_X$ as the Loday algebroid {\it Lie
derivative along $X$}.\medskip

If we define the Lie derivative on the tensor algebra
$T_\R(E)=\bigoplus_{p=0}^\infty\Sec(E)^{\ot_\R p}$ in the obvious
way by
$$\cL_X(X_1\ot_\R\cdots\ot_\R X_p)=\sum_iX_1\ot_\R\dots\ot_\R\stackrel{(i)}
{\overbrace{[X,X_i]}}\ot_\R\dots\ot_\R X_p\,,$$ and if we use the
canonical pairing
$$\la\ell,X_1\ot_\R\dots\ot_\R X_p\ran=\ell(X_1,\dots,X_p)$$ between $\cD^\bullet(E)$ and
$T_\R(E)$, we get \be\label{Ld1} \cL_X\la\ell,X_1\ot_\R\dots\ot_\R
X_p\ran=\la\cL_X\ell,X_1\ot_\R\dots\ot_\R
X_p\ran+\la\ell,\cL_X(X_1\ot_\R\dots\ot_\R X_p)\ran\,.\ee

The following theorem is analogous to the results in the standard
case of a Lie algebroid $E=\sT M$ and operations on the Grassmann
algebra $\zW(M)\subset \cD^\bullet(\sT M)$ of differential forms.

\begin{thm} The graded derivations $\pa$, $i_X$, and $\cL_X$ on $\cD^\bullet(E)$ satisfy the following identities:
\begin{itemize}
\item[(a)] $2\pa^2=[\pa,\pa]_{\ope{sc}}=0$\;, \item[(b)]
$\cL_X=[\pa,i_X]_{\ope{sc}}=\pa i_X+i_X\pa$\;, \item[(c)]
$\pa\cL_X-\cL_X\pa=[\pa,\cL_X]_{\ope{sc}}=0$\;, \item[(d)]
$\cL_Xi_Y-i_Y\cL_X=[\cL_X,i_Y]_{\ope{sc}}=i_{[X,Y]}$\;, \item[(e)]
$\cL_X\cL_Y-\cL_Y\cL_X=[\cL_X,\cL_Y]_{\ope{sc}}=\cL_{[X,Y]}$ .
\end{itemize}
\end{thm}\smallskip

\begin{proof} The results (a), (b), and (c) are obvious. Identity (d) is immediately checked by direct computation. The last
equality is a consequence of (c), (d), and the Jacobi identity
applied to $[\cL_X,[\pa,i_Y]_{\ope{sc}}]_{\ope{sc}}$.\end{proof}

Note that we can easily calculate the Lie derivatives of negative
degrees, $\cL_{X\Box Y}:=[\pa,i_{X\Box Y}]_{\ope{sc}}$,
$\cL_{(X\Box Y)\Box Z}:=[\pa,i_{(X\Box Y)\Box Z}]_{\ope{sc}}$, ...
with the help of the graded Jacobi identity. \medskip

Observe finally that Item (d) of the preceding theorem actually
means that
$$i_{[X,Y]}=[\![i_X,i_Y]\!]_{\pa},$$ where the {\small RHS} is the restriction to interior products of the derived bracket on
$\ope{Der}(\cD^{\bullet}(E),\pitchfork\nolinebreak)$ defined by
the graded Lie bracket $[\cdot,\cdot]_{\ope{sc}}$ and the interior
Lie algebra derivation $[\pa,\cdot]_{\ope{sc}}$ of
$\ope{Der}(\cD^{\bullet}(E),\pitchfork)$ induced by the
homological vector field $\pa.$

\vskip1cm
\noindent Janusz GRABOWSKI\\ Polish Academy of Sciences\\ Institute of
Mathematics\\ \'Sniadeckich 8, P.O. Box 21, 00-956 Warsaw,
Poland\\Email: jagrab@impan.pl \\

\noindent David KHUDAVERDYAN\\University of
Luxembourg\\ Campus Kirchberg, Mathematics Research Unit\\ 6, rue Richard Coudenhove-Kalergi, L-1359 Luxembourg
City, Grand-Duchy of Luxembourg
\\Email: david.khudaverdyan@uni.lu \\

\noindent Norbert PONCIN\\University of Luxembourg\\
Campus Kirchberg, Mathematics Research Unit\\ 6, rue Richard Coudenhove-Kalergi, L-1359 Luxembourg
City, Grand-Duchy of Luxembourg\\Email: norbert.poncin@uni.lu

\end{document}